\documentclass[pdflatex,sn-mathphys-num]{sn-jnl}

\usepackage{graphicx}%
\usepackage{multirow}%
\usepackage{amsmath,amssymb,amsfonts}%
\usepackage{amsthm}%
\usepackage{mathrsfs}%
\usepackage[title]{appendix}%
\usepackage{xcolor}%
\usepackage{textcomp}%
\usepackage{manyfoot}%
\usepackage{booktabs}%
\usepackage{algorithm}%
\usepackage{algorithmicx}%
\usepackage{algpseudocode}%
\usepackage{listings}%
\newcommand{\R}{{\mathbb R}}

\theoremstyle{thmstyleone}%
\newtheorem{theorem}{Theorem}%
\newtheorem{lemma}[theorem]{Lemma} 
\newtheorem{corollary}[theorem]{Corollary}
\theoremstyle{definition}
\usepackage{hyperref}
\theoremstyle{thmstyletwo}%
\newtheorem{example}{Example}%

\theoremstyle{thmstylethree}%

\begin{document}

\title[Hyper-dual group inverse: existence, characterizations, and applications]{Hyper-dual group inverse: existence, characterizations, and applications}

\author[1]{\fnm{Tikesh} \sur{Verma}}\email{tverma.phd2023.maths@nitrr.ac.in}
\equalcont{These authors contributed equally to this work.}
\author[2]{\fnm{Amit} \sur{Kumar}}\email{amit.ss@srmap.edu.in}
\equalcont{These authors contributed equally to this work.}

\author[3]{\fnm{Vaibhav} \sur{Shekhar}}\email{vaibhav.shekhar@cuj.ac.in}
\equalcont{These authors contributed equally to this work.}
\author*[4]{\fnm{N\'estor} \sur{Thome}}\email{njthome@mat.upv.es}
\equalcont{These authors contributed equally to this work.}
\affil[1]{\orgdiv{Department of Mathematics}, \orgname{National Institute of Technology Raipur}, \orgaddress{ \city{Raipur}, \postcode{492010}, \state{Chhattisgarh}, \country{India}}}

\affil[2]{\orgdiv{Department of Mathematics}, \orgname{SRM University-AP}, \orgaddress{ \city{Amravati}, \postcode{522502}, \state{Andhra Pradesh}, \country{India}}}
\affil[3]{\orgdiv{Department of Mathematics}, \orgname{Central University of Jharkhand}, \orgaddress{ \city{Ranchi}, \postcode{835222}, \state{Jharkhand}, \country{India}}}

\affil*[4]{\orgdiv{Instituto Universitario de Matem\'atica Multidisciplinar}, \orgname{ Universitat Polit\`ecnica de Val\`encia}, \orgaddress{\postcode{46022}, \state{Valencia}, \country{Spain}}}

\abstract{Motivated by the recent work of Xiao and Zhong [AIMS Math. 9 (2024), 35125--35150: MR4840882], we propose a generalized inverse for a hyper-dual matrix called hyper-dual group generalized inverse (HDGGI). Firstly, we characterize the existence of the HDGGI of a hyper-dual matrix and we show that it is unique, whenever exists.
 The HDGGI is then used to solve a linear hyper-dual system. We discuss the minimal $P$-norm least-squares properties of hyper-dual group inverse. We also exploit some sufficient conditions under which the reverse and
forward-order laws for a particular form of the HDGGI and hyper-dual Moore-Penrose generalized inverse (HDMPGI) hold. Using the definition of dual matrix of order $n$, we finally establish necessary and sufficient condition for the existence of the group inverse of a dual matrix of order $n$.}



\maketitle

\section{Introduction}
A dual number has a representation of the form $\widehat{a}=a+\epsilon a_0$, where $a$ and $a_0$ are real numbers, and $\epsilon$ is the dual unit satisfying the following properties: $\epsilon\neq 0$ and $\epsilon^2= 0$. In this representation, the numbers associated with $+1$ (i.e. $a$) and the dual unit $\epsilon$ (i.e. $a_0$) are called the {\it primal} and {\it dual} part of $\widehat{a}$, respectively.
Similarly, a hyper-dual number, denoted by $\tilde{a}$, has a representation of the form $\tilde{a}=a_0+\epsilon_1 a_1+\epsilon_2 a_2+\epsilon_1 \epsilon_2 a_3$, where $a_0$, $a_1$, $a_2$  and $a_3$ are real numbers, and the two dual units $\epsilon_1$ and $\epsilon_2$ satisfy the following properties: $\epsilon_1, \epsilon_2, \epsilon_1 \epsilon_2\neq 0$, $\epsilon_1^2= {\epsilon_2}^2=(\epsilon_1\epsilon_2)^2= 0$. A hyper-dual number can be re-written as $\tilde{a}=\widehat{a}+\epsilon_2\widehat{a}_{0}$, where $\widehat{a}=a_0+\epsilon_1 a_1$ and $\widehat{a}_0=a_2+\epsilon_1 a_3$. It is evident from here that a hyper-dual number is nothing but a combination of two dual numbers, where $\widehat{a}$ is called the {\it primal part} and
 $\widehat{a}_0$ is called the {\it hyper-dual part} of $\tilde{a}$, respectively.

The term {\it dual number} was first introduced by William Clifford in the $18^{\text{th}}$ century, while the {\it hyper-dual} was first coined by Fike
{\it et al.} (see \cite{fike1, fike2, fike3}). Both of these numbers have significant applications in mechanics (see \cite{rg1}). The set of dual numbers forms a ring and  the set of hyper-dual numbers forms a commutative, associative, unital algebra over ${\mathbb R}$ ,
denoted by ${\mathbb H}{\mathbb D}$.
An extension of dual (hyper-dual) numbers are dual (hyper-dual) vectors and dual (hyper-dual) matrices where the real numbers are replaced by real vectors and real matrices, respectively. An $n \times n$ real dual matrix $\widehat{A}=A+\epsilon_1 A_0$ is called invertible if the real matrix $A$ is invertible. In this case, the inverse of $\widehat{A}$ 
is given by $\widehat{A}^{-1} = A^{-1} - \epsilon_1 A^{-1}A_0A^{-1}$ \cite{Angeles:1998}. Similarly, an $n\times n$ real hyper-dual matrix $\tilde{A}=\widehat{A}+\epsilon_2 \widehat{A}_0$ (where $\widehat{A}=A_0+\epsilon_1A_1$ and $\widehat{A}_0=A_2+\epsilon_1 A_3$) is said to be invertible if $\widehat{A}^{-1}$ exists, and its inverse is given by $\tilde{A}^{-1}=\widehat{A}^{-1}-\epsilon _2\widehat{A}^{-1}\widehat{A}_0\widehat{A}^{-1}$.
Algebra of such matrices has vast applications in many areas of science and engineering, including kinematic problems \cite{app4}, robotics \cite{angeles}, surface shape analysis and computer graphics \cite{azar}, and rigid body motion \cite{oncu}. Consequently, many researchers are exploring theoretical aspects of dual matrices, for example, the SVD \cite{wei_5}, LU decomposition \cite{wei_1}, the QR decomposition \cite{wei_3}, exponential \cite{wei_2} and dual generalized inverses \cite{wei_4, rg2, mosic_1, dualgroup}.

Generalized inverses are required in order to solve dual linear systems as we can see in many recent papers \cite{ udw2, udw1, dualdrazin, dualgroup}. The study of generalized inverses of a dual matrix allows one to compute the solution of  the dual equation simultaneously. However, unlike the case where the matrix has real entries, the existence of generalized inverses in the case of dual number entries is not guaranteed. So, investigation of the conditions for the existence of generalized inverses and to find their explicit expressions is an important issue. In this aspect, we will now retrieve the literature concerning definitions for some of the generalized inverses and their existence conditions. Let $\widehat{A}$ be a dual matrix of size $m\times n$,  Pennestr\'i and Valentini \cite{rg2, rg1} proposed the Moore-Penrose dual generalized inverse (MPDGI) of $\widehat{A}=A+\epsilon_1 A_0$, which is computed as $\widehat{A}^{p}=A^{\dagger}-\epsilon_1 A^{\dagger}A_0A^{\dagger}$. The MPDGI is used in many inverse problems of kinematics and machines \cite{app3, app4, app2, app1}.
Udwadia {\it et al.} \cite{udw2, udw1} proposed the {\it dual Moore-Penrose generalized inverse} (in short, DMPGI) which is as follows.
Let $\widehat{A}=A+\epsilon_1 A_0$ be a dual matrix, the unique matrix
$\widehat{X}$ (if it exists) satisfying
$$\widehat{A}\widehat{X}\widehat{A}=\widehat{A}, \widehat{X}\widehat{A}\widehat{X}=\widehat{X}, (\widehat{A}\widehat{X})^T=\widehat{A}\widehat{X}, \text{ and } (\widehat{X}\widehat{A})^T=\widehat{X}\widehat{A}$$
is called the DMPGI of $\widehat{A}$, and it is represented as $\widehat{A}^{\dagger}$.
In 2021, Zhong and Zhang \cite{dualgroup} established the {\it dual group generalized inverse} (DGGI). The definition of DGGI is stated next. Let $\widehat{A}=A+\epsilon_1 A_0$ be a dual matrix of size $n\times n$. Then, the unique dual matrix $\widehat{G}$ (if it exists) satisfying
\begin{equation*}
\widehat{A}\widehat{G}\widehat{A}=\widehat{A}, \widehat{G}\widehat{A}\widehat{G}=\widehat{G}, \text{ and } \widehat{A}\widehat{G}=\widehat{G}\widehat{A}
\end{equation*}
is called the DGGI of $\widehat{A}$ and we denote it as $\widehat{A}^{\#}.$ If $A_{0}=0$, then $\widehat{G}$ reduces to the group inverse of the matrix $A$. For more details on the group inverse, its usage in computing other generalized inverses, and its recent developments, see \cite{bask2010, ben1974, zhou_3, zhou_1}. The DGGI is used to find the dual $P$-norm solution of a dual linear system \cite{dualgroup}. Zhong and Zhang \cite{dualdrazin} further established necessary and sufficient conditions for the existence of the {\it dual Drazin generalized inverse} (DDGI) for a dual matrix.
Similarly, several other generalized inverses have been introduced in the literature for dual matrices (see \cite{cui:2024, mosic_2, rg5}).\par

Very recently, Xiao and Zhong \cite{hyper} introduced the notion of {\it hyper-dual Moore-Penrose generalized inverse} (HDMPGI) of a hyper-dual matrix. They further provide characterizations for the existence and least-square properties of the hyper-dual Moore-Penrose generalized inverse, and provided a formula for computing the HDMPGI (if it exists). Motivated by this work, we introduce the {\it hyper-dual group generalized inverse} (HDGGI) of a hyper-dual matrix and study its existence criteria and various characteristics, including the least-squares properties. Inspired by \cite{hyper},  we also obtain the necessary and sufficient condition for the existence of the group inverse for a dual matrix of order $n$.


Let $A$ and $B$ be invertible matrices of the same size. Then the reverse order law and the forward order law are represented as $(AB)^{-1}=B^{-1}A^{-1}$ and $(AB)^{-1}=A^{-1}B^{-1}$, respectively. It is well known that the reverse order law holds, but the forward order law does not always hold. Further, it is known that the reverse order law and forward order law for generalized inverses do not hold, in general.
The reverse and forward order laws for different generalized inverses for elements and matrices are discussed, for example, in \cite{illic_1, illic_2, k11, k12, verma:2024, k13}. But it has yet to be examined for hyper-dual matrices and for some generalized inverse of dual matrices. This article also  presents sufficient conditions for reverse and forward order laws to hold for the HDMPGI and HDGGI.
 
This paper is organized in the following way:  Section \ref{rmm1}  recalls some preliminary results. In Section \ref{rm2}, we introduce hyper-dual group inverse of a hyper-dual matrix. After that, we investigate existence, characterization, and properties of the HDGGI. Section \ref{rm3} discusses some results for solving a hyper-dual linear system using HDGGI, and  establishes sufficient conditions for the minimal P-norm least squares solution. In Section \ref{rm4}, we provide criteria for the reverse and forward order laws for DMPGI and DGGI. Finally, Section \ref{rm5} establishes necessary and sufficient conditions for the existence of
the group inverse of a dual matrix of order $n$.\\

\section{Preliminaries}\label{rmm1}
First, we will fix some of the notation that is frequently used in this article. 
Let $\mathbb{R}^n$, $\mathbb{\widehat{R}}^n$, and $\mathbb{\tilde{R}}^n$ denote the set of $n$-dimensional real vectors, dual vectors, and hyper-dual vectors, respectively (column-wise). Similarly, let $\mathbb{R}^{n\times n}$, $\mathbb{\widehat{R}}^{n\times n}$, and $\mathbb{\tilde{R}}^{n\times n}$ be the set of all real matrices, dual matrices, and hyper-dual matrices of size $n\times n$, respectively. Throughout this article, we deal with real square, real dual and real hyper-dual matrices. Let $Ind(A)$ denote the index of the matrix $A\in \mathbb{R}^{n\times n}$. We say that the dual matrix $\widehat{A}=A+\epsilon_1 A_0$ has {\it dual index 1} if the range spaces of $\widehat{A}^2$ and $\widehat{A}$ coincide \cite{rg5}. Clearly, whenever $A_0=0$, $\widehat{A}\in \widehat{\mathbb{R}}^{n\times n}$ reduces to $A\in \mathbb{R}^{n\times n}$ and in this case the dual index 1 of $\widehat{A}$ coincides with $Ind(A)$.



The following result will be used to prove some of the main results in this article. The first result provides equivalent conditions for the existence of the dual group inverse of a dual matrix.

\begin{theorem}\label{th2.1}\textnormal{\cite[Theorem 3.2]{dualgroup}}\\
    Let $\widehat{A}=A+\epsilon_1 A_{0}$ be a dual matrix such that $Ind(A)=1$. Then, the following statements are equivalent:
    \begin{enumerate}
        \item[(i)] The dual group inverse of $\widehat{A}$ exists;
        \item[(ii)] $\widehat{A}=P\begin{bmatrix}
            C & 0\\
            0 & 0
        \end{bmatrix}P^{-1}+\epsilon_1 P\begin{bmatrix}
            B_{1} & B_{2}\\
            B_{3} & 0
        \end{bmatrix}P^{-1}$, where $P$ and $C$ are nonsingular matrices;
        \item[(iii)] $(I-AA^{\#})A_{0}(I-AA^{\#})=0$.

         Furthermore, if the dual group inverse of $\widehat{A}$ exists, then
         $$\widehat{A}^{\#}={A}^{\#}+\epsilon_1[-{A}^{\#}{A}_0{A}^{\#}+ ({A}^{\#})^2{A}_0(I-{A}{A}^{\#})+(I-{A}{A}^{\#}){A}_0({A}^{\#})^2].$$ Equivalently, 
                           $$\widehat{A}^{\#}=P\begin{bmatrix}
            C^{-1} & 0\\
            0 & 0
        \end{bmatrix}P^{-1}+\epsilon_1 P\begin{bmatrix}
            -C^{-1}B_{1}C^{-1} & C^{-2}B_{2}\\
            B_{3}C^{-2} & 0
        \end{bmatrix}P^{-1}.$$
    \end{enumerate}
\end{theorem}

\begin{theorem}\label{cor3.2}\textnormal{\cite[Corollary 3.2]{dualgroup}}\\
Let $\widehat{A}=A+\epsilon A_0$ be a dual matrix with $Ind(A)=1$. Then, the dual group inverse of $\widehat{A}$ exists and $\widehat{A}^{\#}=A^{\#}-\epsilon A^{\#}A_0A^{\#}$ if and only if $AA^{\#}A_0=A_0AA^{\#}=A_0$.
\end{theorem}

The next result gives characterizations of the dual index of the dual matrix $\widehat{A}$. 
\begin{theorem}\label{dualindex1}\textnormal{\cite[Theorem 2.5 \& Theorem 2.6]{rg5}}\\
    Let $\widehat{A}=A+\epsilon_1 A_0$ be a dual matrix. Then, the following are equivalent 
    \begin{enumerate}
    \item[(i)] $\widehat{A}^{\#}$ exists.
    \item[(ii)] the dual index of $\widehat{A}$ is 1.
        \item[(iii)] $Ind(A) = 1$ and $(I-AA^{\#})A_0(I-AA^{\#})=0$.
    \end{enumerate}
\end{theorem}

The following result establishes equivalent conditions for the existence of the hyper-dual Moore-Penrose generalized inverse of the given hyper-dual matrix $\tilde{A}$.

\begin{theorem}\label{thm2.2}\textnormal{\cite[Theorem 2.1]{hyper}}\\
Let $\tilde{A}=\widehat{A}+\epsilon_2 \widehat{A}_{0}$ be a hyper-dual matrix. Then, the following are equivalent:
\begin{enumerate}
    \item[(i)] The HDMPGI $\tilde{A}^{\dagger}$ of $\tilde{A}$ exists;
    \item[(ii)] $\widehat{A}^{\dagger}$ exists and $(I-\widehat{A}\widehat{A}^{\dagger})\widehat{A}_{0}(I-\widehat{A}^{\dagger}\widehat{A})=0.$

    Furthermore, if the HDMPGI of $\tilde{A}$ exists, then $$\tilde{A}^{\dagger}=\widehat{A}^{\dagger}+\epsilon_2[-\widehat{A}^{\dagger}\widehat{A}_0\widehat{A}^{\dagger}+ (\widehat{A}^{T}\widehat{A})^{\dagger}\widehat{A}^{T}_0(I-\widehat{A}\widehat{A}^{\dagger})+(I-\widehat{A}^{\dagger}\widehat{A})\widehat{A}^{T}_0(\widehat{A}\widehat{A}^{T})^{\dagger}].$$
\end{enumerate}
\end{theorem}

The above result reduces to the following corollary under a special case.
\begin{corollary}\label{cor2.3}
    Let $\tilde{A} = \widehat{A} + \epsilon_2 \widehat{A}_0$ be a hyper-dual matrix. If $\tilde{A}^{\dagger}$ exists and $\widehat{A}^{\dagger}\widehat{A}\widehat{A}_0^T=\widehat{A}_0^T\widehat{A}\widehat{A}^{\dagger}=\widehat{A}_0^T$, then
$$\tilde{A}^{\dagger}=\widehat{A}^{\dagger}-\epsilon_2 \widehat{A}^{\dagger}\widehat{A}_0\widehat{A}^{\dagger}.$$
\end{corollary}

\section{Hyper-dual generalized group inverse}\label{rm2}

In this section, we provide characterization and existence criteria for the hyper-dual group generalized inverse of a hyper-dual matrix. The following lemma uses the equations 
\begin{equation*}
\tilde{A}\tilde{X}\tilde{A}=\tilde{A}, \tilde{X}\tilde{A}\tilde{X}=\tilde{X}, \text{ and } \tilde{A}\tilde{X}=\tilde{X}\tilde{A}
\end{equation*}
and can be easily obtained.

\begin{lemma}\label{lem2.2}
    Let $\tilde{A}=\widehat{A}+\epsilon_2 \widehat{A}_{0}$ be a hyper-dual matrix such that the dual index of $\widehat{A}$ is 1. Then, a hyper-dual matrix $\tilde{X}=\widehat{X}+\epsilon_2 \widehat{X}_{0}$ is the HDGGI of $\tilde{A}$ if and only if  $$\widehat{X}=\widehat{A}^{\#}$$
    and   
    \begin{align*}
          &\widehat{A}\widehat{X}\widehat{A}_{0}+\widehat{A}\widehat{X}_{0}\widehat{A}+\widehat{A}_{0}\widehat{X}\widehat{A}=\widehat{A}_{0},\\
          &\widehat{X}\widehat{A}\widehat{X}_{0}+\widehat{X}\widehat{A}_{0}\widehat{X}+\widehat{X}_{0}\widehat{A}\widehat{X}=\widehat{X}_{0},\\
          &\widehat{A}\widehat{X}_{0}+\widehat{A}_{0}\widehat{X}=\widehat{X}\widehat{A}_{0}+\widehat{X}_{0}\widehat{A}.
    \end{align*}
\end{lemma}
\begin{proof}
    Since the dual index of the dual matrix $\widehat{A}$ is 1, it follows from Theorem \ref{dualindex1} that $\widehat{A}^{\#}$ exists. If $\tilde{X}=\widehat{X}+\epsilon_2 \widehat{X}_{0}$ is the HDGGI of $\tilde{A}$, then $\tilde{A}$ and $\tilde{X}$ must satisfy the equations: 
\begin{equation*}
\tilde{A}\tilde{X}\tilde{A}=\tilde{A}, \tilde{X}\tilde{A}\tilde{X}=\tilde{X}, \text{ and } \tilde{A}\tilde{X}=\tilde{X}\tilde{A}.
\end{equation*}
After some simplifications, comparing the primal parts of the equations, we get
$$\widehat{A}\widehat{X}\widehat{A}=\widehat{A}, \widehat{X}\widehat{A}\widehat{X}=\widehat{X}, \text{ and } \widehat{A}\widehat{X}=\widehat{X}\widehat{A},$$
which shows that $\widehat{X}=\widehat{A}^{\#}$. Similarly, the other equations are obtained by comparing the dual parts of the equations. The converse part can be easily proved.
\end{proof}

 Zhong and Zhang \cite{dualgroup} shown that the group inverse of a dual matrix is unique whenever it exists. We next prove that the hyper-dual group generalized inverse of a hyper-dual matrix is unique whenever it exists. 

\begin{theorem}\label{thm3.2}
    Let $\tilde{A}=\widehat{A}+\epsilon_2\widehat{A}_{0}$ be a hyper-dual matrix such that the dual index of $\widehat{A}$ is 1. If the hyper-dual group generalized inverse of $\tilde{A}$ exists, then it is unique.
\end{theorem}
 \begin{proof}
     According to the Lemma \ref{lem2.2}, if the hyper-dual group inverse of $\tilde{A}=\widehat{A}+\epsilon_2\widehat{A}_{0}$ exists, then it is of the form $\widehat{A}^\#+\epsilon_2 \widehat{C}$. Let $\tilde{Y}_{1}=\widehat{A}^{\#}+\epsilon_2\widehat{C}_{1}$ and $\tilde{Y}_{2}=\widehat{A}^\#+\epsilon_2\widehat{C}_{2}$ be two hyper-dual group inverses of $\tilde{A}$. To establish the uniqueness of $\tilde{A}^\#$, it suffices to show that $\widehat{C}_{1}=\widehat{C}_{2}$.
     Equating the dual part of $\tilde{A}\tilde{Y}_{1}\tilde{A}=\tilde{A}$ and $\tilde{A}\tilde{Y}_{2}\tilde{A}=\tilde{A}$, we get 
     \begin{equation}\label{eq3.1}
         \widehat{A}\widehat{A}^\#\widehat{A}_{0}+\widehat{A}\widehat{C}_{1}\widehat{A}+\widehat{A}_{0}\widehat{A}^\#\widehat{A}=\widehat{A}_{0}
      \end{equation}   
      and 
      \begin{equation}\label{eq3.2}
        \widehat{A}\widehat{A}^\#\widehat{A}_{0}+\widehat{A}\widehat{C}_{2}\widehat{A}+\widehat{A}_{0}\widehat{A}^\#\widehat{A}=\widehat{A}_{0}.
      \end{equation}
      From equations~\eqref{eq3.1} and \eqref{eq3.2}, we have
      \begin{equation}\label{eq3.3}
          \widehat{A}(\widehat{C}_{1}-\widehat{C}_{2})\widehat{A}=0.
      \end{equation}
      Similarly, equating the dual part of $\tilde{Y}_{1}\tilde{A}\tilde{Y}_{1}=\tilde{Y}_{1}$ and $\tilde{Y}_{2}\tilde{A}\tilde{Y}_{2}=\tilde{Y}_{2}$, we get
      \begin{equation}\label{eq3.4}
          \widehat{A}^\#\widehat{A}\widehat{C}_{1}+\widehat{A}^\#\widehat{A}_{0}\widehat{A}^\#+\widehat{C}_{1}\widehat{A}\widehat{A}^\#=\widehat{C}_{1}
      \end{equation}
     and 
     \begin{equation}\label{eq3.5}
\widehat{A}^\#\widehat{A}\widehat{C}_{2}+\widehat{A}^\#\widehat{A}_{0}\widehat{A}^\#+\widehat{C}_{2}\widehat{A}\widehat{A}^\#=\widehat{C}_{2}.
      \end{equation}
      Subtracting equation~\eqref{eq3.5} from \eqref{eq3.4}, we get
      \begin{equation}\label{eq3.6}
          \widehat{C}_{1}-\widehat{C}_{2}=\widehat{A}^\#\widehat{A}(\widehat{C}_{1}-\widehat{C}_{2})+(\widehat{C}_{1}-\widehat{C}_{2})\widehat{A}\widehat{A}^{\#}.
      \end{equation}
      Now, equating the dual part of $\tilde{A}\tilde{Y}_{1}=\tilde{Y}\tilde{A}_{1}$ and $\tilde{A}\tilde{Y}_{2}=\tilde{Y}_{2}\tilde{A}$, we obtain
      \begin{equation}\label{eq3.7}
          \widehat{A}\widehat{C}_{1}+\widehat{A}_{0}\widehat{A}^\#=\widehat{C}_{1}\widehat{A}+\widehat{A}^{\#}\widehat{A}_{0}
      \end{equation}
      and 
      \begin{equation}\label{eq3.8}
          \widehat{A}\widehat{C}_{2}+\widehat{A}_{0}\widehat{A}^\#=\widehat{C}_{2}\widehat{A}+\widehat{A}^{\#}\widehat{A}_{0}.
      \end{equation}
      From equations~\eqref{eq3.7} and \eqref{eq3.8}, we have
      \begin{equation}\label{eq3.9}
          \widehat{A}(\widehat{C}_{1}-\widehat{C}_{2})=(\widehat{C}_{1}-\widehat{C}_{2})\widehat{A}.
      \end{equation}
      Now, post-multiplying Equation~\eqref{eq3.3} by $\widehat{A}^{\#}$ gives $\widehat{A}(\widehat{C}_{1}-\widehat{C}_{2})\widehat{A}\widehat{A}^{\#}=0$. Thus, it follows from Equation~\eqref{eq3.9} that 
      $0=\widehat{A}(\widehat{C}_{1}-\widehat{C}_{2})\widehat{A}\widehat{A}^{\#}=(\widehat{C}_{1}-\widehat{C}_{2})\widehat{A}\widehat{A}\widehat{A}^{\#}=(\widehat{C}_{1}-\widehat{C}_{2})\widehat{A}=\widehat{A}(\widehat{C}_{1}-\widehat{C}_{2})$, which implies from \eqref{eq3.6}
      $$ \widehat{C}_{1}-\widehat{C}_{2}=\widehat{A}^\#\widehat{A}(\widehat{C}_{1}-\widehat{C}_{2})+(\widehat{C}_{1}-\widehat{C}_{2})\widehat{A}\widehat{A}^{\#}=0,$$
      i.e, $\widehat{C}_{1}=\widehat{C}_{2}$.
 \end{proof}

We next provide a characterization for the existence of the group inverse of a hyper-dual matrix and present a compact formula for its computation. 
\begin{theorem}
    \label{th2.4}
    Let $\tilde{A}=\widehat{A}+\epsilon_2\widehat{A}_{0}=A_{0}+\epsilon_1 A_{1}+\epsilon_2A_{2}+\epsilon_1 \epsilon_2A_{3}$ be such that the dual index of $\widehat{A}$ is 1. Then, the following statements are equivalent:
    \begin{enumerate}
    
        \item[(i)] The HDGGI of $\tilde{A}$ exists;
        \item[(ii)] $\tilde{A}=P\begin{bmatrix}
            C & 0\\
            0 & 0
        \end{bmatrix}P^{-1}+\epsilon_1 P\begin{bmatrix}
            B_{1} & B_{2}\\
            B_{3} & 0
        \end{bmatrix}P^{-1}+\epsilon_2\Bigg(P\begin{bmatrix}
            Y_{1} & Y_{2}\\
            Y_{3} & 0
        \end{bmatrix}P^{-1}+\epsilon_1 P\begin{bmatrix}
            Z_{1} & Z_{2}\\
            Z_{3} & Z_{4}
        \end{bmatrix}P^{-1}\Bigg)$,
        where $P$ and $C$ are nonsingular matrices of orders n and r, respectively, and $B_{{i}'}s$, $Y_{{i}'}s$, $Z_{{i}'}s$ are real matrices of appropriate sizes that satisfy
            $$Z_{4}=B_{3}C^{-1}Y_{2}+Y_{3}C^{-1}B_{2};$$
        \item[(iii)] $(I-\widehat{A}\widehat{A}^{\#})\widehat{A}_{0}(I-\widehat{A}\widehat{A}^{\#})=0;$  
        \item[(iv)] \begin{align*}
        (I-A_{0}A_{0}^{\#})A_{1}(I-A_{0}A_{0}^{\#})&= (I-A_{0}A_{0}^{\#})A_{2}(I-A_{0}A_{0}^{\#})\\
                                                       &= {(I-A_{0}A_{0}^{\#})(A_{3}-A_{1}A_{0}^{\#}A_{2}-A_{2}A_{0}^{\#}A_{1})(I-A_{0}A_{0}^{\#})}=0. 
                                                   \end{align*}
        Furthermore, if the hyper-dual group inverse of 
        $\tilde{A}$ exists, then 
        $$\tilde{A}^{\#}=\widehat{A}^{\#}+\epsilon_2(-\widehat{A}^{\#}\widehat{A}_{0}\widehat{A}^{\#}+(\widehat{A}^{\#})^{2}\widehat{A}_{0}(I-\widehat{A}\widehat{A}^{\#})+(I-\widehat{A}\widehat{A}^{\#})\widehat{A}_{0}(\widehat{A}^{\#})^{2}).$$
    \end{enumerate}
\end{theorem}
\begin{proof}
    $(i) \implies (ii):$\\ Suppose that the HDGGI of $$\tilde{A}=\widehat{A}+\epsilon_2 \widehat{A}_{0}$$ exists, and has the form  
$$\tilde{A}^{\#}=\widehat{X}+\epsilon_2 \widehat{X}_{0}.$$ Then, it follows from Lemma~\ref{lem2.2} that DGGI of $\widehat{A}$ exists and $\widehat{X}=\widehat{A}^{\#}$. Since the dual index of $\widehat{A}$ is 1, consider the dual matrix $\widehat{A}=A_0+\epsilon_1 A_1$, then by Theorem~\ref{dualindex1}, we have $Ind(A_0)=1$. Now, applying Theorem~\ref{th2.1} on the dual matrix $\widehat{A}=A_0+\epsilon_1 A_1$, we get
       $$\widehat{A}=P\begin{bmatrix}
            C & 0\\
            0 & 0
        \end{bmatrix}P^{-1}+\epsilon_1 P\begin{bmatrix}
            B_{1} & B_{2}\\
            B_{3} & 0
        \end{bmatrix}P^{-1}$$ and 
        $$\widehat{X}=\widehat{A}^{\#}=P\begin{bmatrix}
            C^{-1} & 0\\
            0 & 0
        \end{bmatrix}P^{-1}+\epsilon_1 P\begin{bmatrix}
            -C^{-1}B_{1}C^{-1} & C^{-2}B_{2}\\
            B_{3}C^{-2} & 0
        \end{bmatrix}P^{-1},$$
        where $P$ and $C$ are nonsingular matrices and $B_{1}, B_{2}, B_{3}$ are real matrices of suitable sizes. Let 
        \begin{align*}
            \widehat{A}_{0}&=P\begin{bmatrix}
            Y_{1} & Y_{2}\\
            Y_{3} & Y_{4}
        \end{bmatrix}P^{-1}+\epsilon_1 P\begin{bmatrix}
            Z_{1} & Z_{2}\\
            Z_{3} & Z_{4}
        \end{bmatrix}P^{-1}\end{align*}
        and
        \begin{align*}
\widehat{X}_{0}&=P\begin{bmatrix}
            X_{1} & X_{2}\\
            X_{3} & X_{4}
        \end{bmatrix}P^{-1}+\epsilon_1P\begin{bmatrix}
            W_{1} & W_{2}\\
            W_{3} & W_{4}
        \end{bmatrix}P^{-1}.
        \end{align*}
        Then, after some computations, we have 
           \begin{align*}
               \widehat{A}\widehat{A}^{\#}\widehat{A}_{0}&=P\begin{bmatrix}
            Y_{1} & Y_{2}\\
            0 & 0
        \end{bmatrix}P^{-1}+\epsilon_1 P\begin{bmatrix}
            Z_{1}+C^{-1}B_{2}Y_{3} & Z_{2}+C^{-1}B_{2}Y_{4}\\
            B_{3}C^{-1}Y_{1} & B_{3}C^{-1}Y_{2}
        \end{bmatrix}P^{-1},\\
        \widehat{A}\widehat{X}_{0}\widehat{A}&=P\begin{bmatrix}
            CX_{1}C & 0\\
            0 & 0
        \end{bmatrix}P^{-1}+\epsilon_1 P\begin{bmatrix}
            \theta & CX_{1}B_{2}\\
            B_{3}X_{1}C & 0
        \end{bmatrix}P^{-1},
           \end{align*}
           where $\theta=CW_{1}C+B_{1}X_{1}C+B_{2}X_{3}C+CX_{1}B_{1}+CX_{2}B_{3}$,
           and $$\widehat{A}_{0}\widehat{A}^{\#}\widehat{A}=P\begin{bmatrix}
            Y_{1} & 0\\
            Y_{3} & 0
        \end{bmatrix}P^{-1}+\epsilon_1 P\begin{bmatrix}
            Z_{1}+Y_{2}B_{3}C^{-1} & Y_{1}C^{-1}B_{2}\\
            Z_{3}+Y_{4}B_{3}C^{-1} & Y_{3}C^{-1}B_{2}
        \end{bmatrix}P^{-1}.$$
        Thus,    
          $$\widehat{A}\widehat{A}^{\#}\widehat{A}_{0}+\widehat{A}\widehat{X}_{0}\widehat{A}+\widehat{A}_{0}\widehat{A}^{\#}\widehat{A}=P\begin{bmatrix}
            2Y_{1}+CX_{1}C & Y_{2}\\
            Y_{3} & 0
        \end{bmatrix}P^{-1}+\epsilon_1 P\begin{bmatrix}
            \gamma_{1} & \gamma_{2}\\
            \gamma_{3} & \gamma_{4}
        \end{bmatrix}P^{-1},$$
        where \begin{align*}
            \gamma_{1}&=2Z_{1}+C^{-1}B_{2}Y_{3}+Y_{2}B_{3}C^{-1}+\theta,\\
            \gamma_{2}&=Z_{2}+C^{-1}B_{2}{Y_{4}}+CX_{1}B_{2}+Y_{1}C^{-1}B_{2},\\
            \gamma_{3}&=Z_{3}+B_{3}C^{-1}Y_{1}+B_{3}X_{1}C+{Y_4}B_{3}C^{-1},\\
            \gamma_{4}&=B_{3}C^{-1}Y_{2}+Y_{3}C^{-1}B_{2}.
        \end{align*}
        Now, from Lemma \ref{lem2.2}, we have 
                          
  $$ P\begin{bmatrix}
            2Y_{1}+CX_{1}C & Y_{2}\\
            Y_{3} & 0
        \end{bmatrix}P^{-1}+\epsilon_1 P\begin{bmatrix}
            \gamma_{1} & \gamma_{2}\\
            \gamma_{3} & \gamma_{4}
        \end{bmatrix}P^{-1}=P\begin{bmatrix}
            Y_{1} & Y_{2}\\
            Y_{3} & Y_{4}
        \end{bmatrix}P^{-1}+\epsilon_1 P\begin{bmatrix}
            Z_{1} & Z_{2}\\
            Z_{3} & Z_{4}
        \end{bmatrix}P^{-1}.$$
        Comparing the primal part and the dual part of both sides of the above equality gives $Y_{4}=0$ and 
        $$\gamma_{4}=Z_{4}=B_{3}C^{-1}Y_{2}+Y_{3}C^{-1}B_{2}.$$
     Therefore, $\tilde{A}$ has the form 
     $$\tilde{A}=P\begin{bmatrix}
            C & 0\\
            0 & 0
        \end{bmatrix}P^{-1}+\epsilon_1 P\begin{bmatrix}
            B_{1} & B_{2}\\
            B_{3} & 0
        \end{bmatrix}P^{-1}+\epsilon_2\Bigg(P\begin{bmatrix}
            Y_{1} & Y_{2}\\
            Y_{3} & 0
        \end{bmatrix}P^{-1}+\epsilon_1 P\begin{bmatrix}
            Z_{1} & Z_{2}\\
            Z_{3} & Z_{4}
        \end{bmatrix}P^{-1}\Bigg).$$
   $(ii) \implies (iii):$\\
Since the dual index of $\widehat{A}$ is 1, by Theorem \ref{dualindex1}, $Ind(A_0)=1$. Let $\tilde{A}=\widehat{A}+\epsilon_2 \widehat{A}_{0}$,
   where  
$$\widehat{A}=P\begin{bmatrix}
            C & 0\\
            0 & 0
        \end{bmatrix}P^{-1}+\epsilon_1 P\begin{bmatrix}
            B_{1} & B_{2}\\
            B_{3} & 0
        \end{bmatrix}P^{-1} ~~\text{ and }~~\widehat{A}_{0}=P\begin{bmatrix}
            Y_{1} & Y_{2}\\
            Y_{3} & 0
        \end{bmatrix}P^{-1}+\epsilon_1 P\begin{bmatrix}
            Z_{1} & Z_{2}\\
            Z_{3} & Z_{4}
        \end{bmatrix}P^{-1}.$$
        Then, from Theorem \ref{th2.1}, the DGGI of $\widehat{A}$ exists. Now, 
\begin{align*}
            \Bigg((I_n - \widehat{A}\widehat{A}^{\#}) \widehat{A}_0 (I_n - \widehat{A}^{\#} \widehat{A})\Bigg) &=\Bigg(P\begin{bmatrix}
                0 & 0\\
                0 & I_{n-r}
            \end{bmatrix}P^{-1}+\epsilon_1 P\begin{bmatrix}
                    0 & -C^{-1}B_{2}\\
                    -B_{3}C^{-1} & 0
                \end{bmatrix}P^{-1} \Bigg)\\ &~~~~\times \Bigg(P\begin{bmatrix}
            Y_{1} & Y_{2}\\
            Y_{3} & 0
        \end{bmatrix}P^{-1}+\epsilon_1 P\begin{bmatrix}
            Z_{1} & Z_{2}\\
            Z_{3} & Z_{4}
        \end{bmatrix}P^{-1}\Bigg)\\ 
        &~~~~\times \Bigg(P\begin{bmatrix}
                0 & 0\\
                0 & I_{n-r}
            \end{bmatrix}P^{-1}+\epsilon_1 P\begin{bmatrix}
                    0 & -C^{-1}B_{2}\\
                    -B_{3}C^{-1} & 0
                \end{bmatrix}P^{-1} \Bigg)\\
                &=\epsilon_1 P\begin{bmatrix}
                    0 & 0\\
                    0 & Z_{4}-(B_{3}C^{-1}Y_{2}+Y_{3}C^{-1}B_{2})
                \end{bmatrix}P^{-1}.
\end{align*}
Therefore, if $Z_{4}=B_{3}C^{-1}Y_{2}+Y_{3}C^{-1}B_{2}$, then 
$$(I-\widehat{A}\widehat{A}^{\#})\widehat{A}_{0}(I-\widehat{A}\widehat{A}^{\#})=0.$$
$(iii) \implies (i):$\\
As dual index of $\widehat{A}$ is 1, by Theorem \ref{dualindex1}, DGGI of $\widehat{A}$ exists. Now, applying Theorem \ref{th2.1} on $\widehat{A}$, we have 
                       \begin{align*}
                    \widehat{A}&=P\begin{bmatrix}
            C & 0\\
            0 & 0
        \end{bmatrix}P^{-1}+\epsilon_1 P\begin{bmatrix}
            B_{1} & B_{2}\\
            B_{3} & 0
        \end{bmatrix}P^{-1}
        \end{align*}
             and  
\begin{align*}\widehat{A}^{\#}&=P\begin{bmatrix}
            C^{-1} & 0\\
            0 & 0
        \end{bmatrix}P^{-1}+{\epsilon_1}P\begin{bmatrix}
            -C^{-1}B_{1}C^{-1} &~~C^{-2}B_{2}\\
            B_{3}C^{-2} & 0
        \end{bmatrix}P^{-1}.
                       \end{align*}
                       If $\widehat{A}_{0}=P\begin{bmatrix}
            Y_{1} & Y_{2}\\
            Y_{3} & Y_{4}
        \end{bmatrix}P^{-1}+\epsilon_1 P\begin{bmatrix}
            Z_{1} & Z_{2}\\
            Z_{3} & Z_{4}
        \end{bmatrix}P^{-1}$, then it follows from $(I-\widehat{A}\widehat{A}^{\#})\widehat{A}_{0}(I-\widehat{A}\widehat{A}^{\#})=0$ that 
           $$P\begin{bmatrix}
               0 & 0\\
               0 & Y_{4}
           \end{bmatrix}P^{-1}+\epsilon_1 P\begin{bmatrix}
               0 & -C^{-1}B_{2}Y_{4}\\
               -Y_{4}B_{3}C^{-1} &~~Z_{4}-(B_{3}C^{-1}Y_{2}+Y_{3}C^{-1}B_{2})
           \end{bmatrix}P^{-1}=0.$$
           It can be seen from the above equality that $Y_{4}=0$ and $Z_{4}=B_{3}C^{-1}Y_{2}+Y_{3}C^{-1}B_{2}$. Now, define
           \begin{align}\label{eq3.10}
               \tilde{G}&=P\begin{bmatrix}
            C^{-1} & 0\\
            0 & 0
        \end{bmatrix}P^{-1}+\epsilon_1 P\begin{bmatrix}
            -C^{-1}B_{1}C^{-1} &~~C^{-2}B_{2}\\
            B_{3}C^{-2} & 0
        \end{bmatrix}P^{-1}\notag\\
        &+\epsilon_2\Bigg(P\begin{bmatrix}
            -C^{-1}Y_{1}C^{-1} & C^{-2}Y_{2}\\
            Y_{3}C^{-2} & 0
        \end{bmatrix}P^{-1}+\epsilon_1 P\begin{bmatrix}
            M_{1} & M_{2}\\
            M_{3} & M_{4}
        \end{bmatrix}P^{-1}\Bigg),
           \end{align}
        where  
        \begin{align*}
           M_{1}&=-C^{-1}B_{2}Y_{3}C^{-2}-C^{-2}Y_{2}B_{3}C^{-1}-C^{-1}Z_{1}C^{-1}+C^{-1}B_{1}C^{-1}Y_{1}C^{-1}-C^{-2}B_{2}Y_{3}C^{-1}\\
           &+C^{-1}Y_{1}C^{-1}B_{1}C^{-1}-C^{-1}Y_{2}B_{3}C^{-2},  \\
           M_{2}&=C^{-2}Z_{2}-C^{-2}B_{1}C^{-1}Y_{2}-C^{-2}Y_{1}C^{-1}B_{2}-C^{-1}B_{1}C^{-2}Y_{2}-C^{-1}Y_{1}C^{-2}B_{2},\\
           M_{3}&=Z_{3}C^{-2}-B_{3}C^{-1}Y_{1}C^{-2}-Y_{3}C^{-1}B_{1}C^{-2}-Y_{3}C^{-2}B_{1}C^{-1}-B_{3}C^{-2}Y_{1}C^{-1},\\
           M_{4}&=B_{3}C^{-3}Y_{2}+Y_{3}C^{-3}B_{2}.
        \end{align*}
    Then, the  primal part of $\tilde{A}\tilde{G}$ is given by
                   $$\widehat{A}^{\#}\widehat{A}=P\begin{bmatrix}
            I & 0\\
            0 & 0
        \end{bmatrix}P^{-1}+\epsilon_1 P\begin{bmatrix}
            0 & C^{-1}B_{2}\\
            B_{3}C^{-1} & 0
        \end{bmatrix}P^{-1}$$ while the hyper-dual part of $\tilde{A}\tilde{G}$ is given as 
             \begin{align*}
                 \widehat{A}\widehat{X}_{0}+\widehat{A}_{0}\widehat{A}^{\#}&=P\begin{bmatrix}
             0 & C^{-1}Y_{2}\\
            Y_{3}C^{-1} & 0
        \end{bmatrix}P^{-1}+\epsilon_1 P\begin{bmatrix}
            P_{1} & P_{2}\\
            P_{3} & P_{4}
        \end{bmatrix}P^{-1},
             \end{align*}
        where \begin{align*}
            P_{1}&=-C^{-1}B_2Y_{3}C^{-1}-C^{-1}Y_{2}B_{3}C^{-1},\\
            P_{2}&=C^{-1}Z_{2}-C^{-1}B_{1}C^{-1}Y_{2}-C^{-1}Y_{1}C^{-1}B_{2},\\
            P_{3}&=Z_{3}C^{-1}-Y_{3}C^{-1}B_{1}C^{-1}-{B_3}C^{-1}Y_{1}C^{-1},\\
   P_{4}&=Y_{3}C^{-2}B_{2}+B_{3}C^{-2}Y_{2}.
        \end{align*}
        Similarly, the primal part of $\tilde{G}\tilde{A}$ is given by
         $$\widehat{A}\widehat{A}^{\#}=P\begin{bmatrix}
            I & 0\\
            0 & 0
        \end{bmatrix}P^{-1}+\epsilon_1 P\begin{bmatrix}
            0 & C^{-1}B_{2}\\
            B_{3}C^{-1} & 0
        \end{bmatrix}P^{-1}$$ while the hyper-dual part of $\tilde{G}\tilde{A}$ is 
             \begin{align*}
                 \widehat{A}^{\#}\widehat{A}_{0}+\widehat{X}_{0}\widehat{A}&=P\begin{bmatrix}
             0 & C^{-1}Y_{2}\\
            Y_{3}C^{-1} & 0
        \end{bmatrix}P^{-1}+\epsilon_1 P\begin{bmatrix}
            N_{1} & N_{2}\\
            N_{3} & N_{4}
        \end{bmatrix}P^{-1},
             \end{align*}
        where \begin{align*}
            N_{1}&=-C^{-1}B_2Y_{3}C^{-1}-C^{-1}Y_{2}B_{3}C^{-1},\\
            N_{2}&=C^{-1}Z_{2}-C^{-1}B_{1}C^{-1}Y_{2}-C^{-1}Y_{1}C^{-1}B_{2},\\
            N_{3}&=Z_{3}C^{-1}-Y_{3}C^{-1}B_{1}C^{-1}-{B_3}C^{-1}Y_{1}C^{-1},\\
            N_{4}&=Y_{3}C^{-2}B_{2}+B_{3}C^{-2}Y_{2}.
        \end{align*}
    Thus, $\tilde{A}\tilde{G}=\tilde{G}\tilde{A}$.
     Now, {\small{\begin{align*}
         \tilde{A}\tilde{G}\tilde{A}&=P\begin{bmatrix}
            C & 0\\
            0 & 0
        \end{bmatrix}P^{-1}+\epsilon_1 P\begin{bmatrix}
            0 & B_{2}\\
            0 & 0
        \end{bmatrix}P^{-1}+\epsilon_1 P\begin{bmatrix}
            B_{1} & 0\\
            B_{3} & 0
        \end{bmatrix}P^{-1}\\
        &+\epsilon_2\Bigg(P\begin{bmatrix}
            0 & Y_{2}\\
            0 & 0
        \end{bmatrix}P^{-1}+\epsilon_1 P\begin{bmatrix}
            B_{2}Y_{3}C^{-1} & B_{1}C^{-1}Y_{2}\\
            0 & B_{3}C^{-1}Y_{2}
        \end{bmatrix}P^{-1}\\
        &+\epsilon_1 P\begin{bmatrix}
            -B_{2}Y_{3}C^{-1}-Y_{2}B_{3}C^{-1} & Z_{2}-B_{1}C^{-1}Y_{2}-Y_{1}C^{-1}B_{2}\\
            0 & 0
        \end{bmatrix}P^{-1}+P\begin{bmatrix}
            Y_{1} & 0\\
            Y_{3} & 0
        \end{bmatrix}P^{-1}\\
        &+\epsilon_1 P\begin{bmatrix}
            Y_{2}B_{3}C^{-1} & Y_{1}C^{-1}B_{2}\\
            0 & Y_{3}C^{-1}B_{2}
        \end{bmatrix}P^{-1}+\epsilon_1 P \begin{bmatrix}
            Z_{1} & 0\\
            Z_{2} & 0
        \end{bmatrix}P^{-1}\Bigg)\\
        & = P\begin{bmatrix}
            C & 0\\
            0 & 0
        \end{bmatrix}P^{-1}+\epsilon_1 P\begin{bmatrix}
            B_{1} & B_{2}\\
            B_{3} & 0
        \end{bmatrix}P^{-1}+\epsilon_2\Bigg(P\begin{bmatrix}
            Y_{1} & Y_{2}\\
            Y_{3} & 0
        \end{bmatrix}P^{-1}+\epsilon_1 P\begin{bmatrix}
            Z_{1} & Z_{2}\\
            Z_{3} & Z_{4}
        \end{bmatrix}P^{-1}\Bigg)\\
        &=\tilde{A}.
     \end{align*}}}
     Again, {\small{\begin{align*}
         \tilde{G}\tilde{A}\tilde{G}&=P\begin{bmatrix}
             C^{-1} & 0\\
             0 & 0
         \end{bmatrix}P^{-1}+\epsilon_1 P\begin{bmatrix}
             0 & C^{-2}B_{2}\\
             0 & 0
         \end{bmatrix}P^{-1}+\epsilon_1 P \begin{bmatrix}
             -C^{-1}B_{1}C^{-1} & 0\\
             B_{2}C^{-2} & 0
         \end{bmatrix}P^{-1}+\epsilon_2 \Bigg(P \begin{bmatrix}
             0 & C^{-2}B_{2}\\
             0 & 0
         \end{bmatrix}P^{-1}\\
         &+\epsilon_1 P\begin{bmatrix}
             -C^{-2}B_{2}Y_{3}C^{-1}-C^{-2}Y_{2}B_{3}C^{-1} & C^{-2}Z_{2}-C^{-2}B_{1}C^{-1}Y_{2}-C^{-2}Y_{1}C^{-1}B_{2}\\
             0 & 0
         \end{bmatrix}P^{-1}\\
         &+\epsilon_1 P \begin{bmatrix}
             C^{-2}B_{2}Y_{3}C^{-1} & -C^{-1}B_{1}C^{-2}Y_{2}\\
             0 & B_{3}C^{-3}Y_{2}
         \end{bmatrix}P^{-1}+P \begin{bmatrix}
             -C^{-1}Y_{1}C^{-1} & 0\\
             Y_{3}C^{-2} & 0 
         \end{bmatrix}P^{-1} + \epsilon_1 P \begin{bmatrix}
             M_{1} & 0\\
             M_{2} & 0
         \end{bmatrix}P^{-1}\\
         &+\epsilon_1 P\begin{bmatrix}
             C^{-2}Y_{2}B_{3}C^{-1} & -C^{-1}Y_{1}C^{-2}B_{2}\\
             0 & Y_{3}C^{-3}B_{2}
         \end{bmatrix}P^{-1}\Bigg)\\
         &= P\begin{bmatrix}
            C^{-1} & 0\\
            0 & 0
        \end{bmatrix}P^{-1}+P\begin{bmatrix}
            -C^{-1}B_{1}C^{-1} & C^{-2}B_{2}\\
            B_{3}C^{-2} & 0
        \end{bmatrix}P^{-1}+\epsilon_2\Bigg(P\begin{bmatrix}
            -C^{-1}Y_{1}C^{-1} & C^{-2}Y_{2}\\
            Y_{3}C^{-2} & 0
        \end{bmatrix}P^{-1}\\
        &+\epsilon_1 P\begin{bmatrix}
            M_{1} & M_{2}\\
            M_{3} & M_{4}
        \end{bmatrix}P^{-1}\Bigg)\\
        &=\tilde{G}.
     \end{align*}}}   
     Thus, $\tilde{A}^{\#}$ exists and $\tilde{A}^{\#}=\tilde{G}$.\\
     $(iii) \iff (iv)$:\\
Since $\widehat{A}^{\#}$ exists, from Theorem \ref{th2.1}, we have 
      $$(I-A_{0}A_{0}^{\#})A_{1}(I-A_{0}A_{0}^{\#})=0.$$
      If $(I-\widehat{A}\widehat{A}^{\#})\widehat{A}_{0}(I-\widehat{A}\widehat{A}^{\#})=0$, then substituting $\widehat{A}=A_{0}+\epsilon_1 A_{1}$, $\widehat{A}_{0}=A_{2}+\epsilon_1 \widehat{A}_{3}$, and 
      $$\widehat{A}^{\#}=A_{0}^{\#}+\epsilon_1[-A_{0}^{\#}A_{1}A_{0}^{\#}+(A_{0}^{\#})^2A_{1}(I-A_{0}A_{0}^{\#})+(I-A_{0}A_{0}^{\#})A_{1}(A_{0}^{\#})^2]$$
      into $$(I-\widehat{A}\widehat{A}^{\#})\widehat{A}_{0}(I-\widehat{A}\widehat{A}^{\#})=0$$ gives 
      \begin{align*}
          (I-A_{0}A_{0}^{\#})A_{2}(I-A_{0}A_{0}^{\#})+\epsilon_1 [(I-A_{0}A_{0}^{\#})(A_{3}-A_{1}A_{0}^{\#}A_{2}-A_{2}A_{0}^{\#}A_{1})(I-A_{0}A_{0}^{\#}\\
      -(I-A_{0}A_{0}^{\#})A_{2}(I-A_{0}A_{0}^{\#})A_{1}A_{0}^{\#}-A_{0}^{\#}A_{1}(I-A_{0}A_{0}^{\#})A_{2}(I-A_{0}A_{0}^{\#})]=0,
      \end{align*}
      which implies  $$ (I-A_{0}A_{0}^{\#})A_{2}(I-A_{0}A_{0}^{\#})=0$$ and
      $${(I-A_{0}A_{0}^{\#})(A_{3}-A_{1}A_{0}^{\#}A_{2}-A_{2}A_{0}^{\#}A_{1})(I-A_{0}A_{0}^{\#})}=0.$$
     Conversely, if $(I-A_{0}A_{0}^{\#})A_{1}(I-A_{0}A_{0}^{\#})=0$, then $\widehat{A}^{\#}$ exists. Moreover, if 
     $$(I-A_{0}A_{0}^{\#})A_{2}(I-A_{0}A_{0}^{\#})=0$$ and 
     $${(I-A_{0}A_{0}^{\#})(A_{3}-A_{1}A_{0}^{\#}A_{2}-A_{2}A_{0}^{\#}A_{1})(I-A_{0}A_{0}^{\#})}=0,$$ then it is easy to see that  $$(I-\widehat{A}\widehat{A}^{\#})\widehat{A}_{0}(I-\widehat{A}\widehat{A}^{\#})=0.$$
     Furthermore,      
     \begin{equation}\label{3.11}
         \widehat{A}^{\#}\widehat{A}_{0}\widehat{A}^{\#}=P\begin{bmatrix}
         C^{-1}Y_{1}C^{-1} & 0\\
         0 & 0
     \end{bmatrix}P^{-1}+\epsilon_1 P \begin{bmatrix}
         T_{1} & C^{-1}Y_{1}C^{-2}B_{2}\\
         B_{3}C^{-2}Y_{-1}C^{-1} & 0
     \end{bmatrix}P^{-1},
     \end{equation}
     where $T_{1}=C^{-1}Z_{1}C^{-1}-C^{-1}B_{1}C^{-1}Y_{1}C^{-1}+C^{-2}B_{2}Y_{3}C^{-1}-C^{-1}Y_{1}C^{-1}B_{-1}C^{-1}+C^{-1}Y_{2}B_{3}C^{-2}$ and
     \begin{equation}\label{3.12}
         (\widehat{A}^{\#})^{2}\widehat{A}_{0}(I-\widehat{A}\widehat{A}^{\#})=P\begin{bmatrix}
         0 &  {C^{-2}Y_2} \\
         0 & 0
     \end{bmatrix}P^{-1}+\epsilon_1 P\begin{bmatrix}
         -C^{2}Y_{2}B_{3}C^{-1} & T_{2}\\
         0 & B_{3}C^{-3}Y_{2}
     \end{bmatrix}P^{-1},
     \end{equation}
     where $T_{2}=C^{-2}Z_{2}-C^{-2}Y_{1}C^{-1}B_{2}-C^{-2}B_{1}C^{-1}Y_{2}-C^{-1}B_{1}C^{-2}Y_{2}.$ Again, 

     \begin{equation}\label{3.13}
         (I-\widehat{A}\widehat{A}^{\#})\widehat{A}_{0}(\widehat{A}^{\#})^{2}=P\begin{bmatrix}
         0 & 0\\
         Y_{3}C^{-2} & 0
     \end{bmatrix}P^{-1}+\epsilon_1 P\begin{bmatrix}
         -C^{-1}B_{2}Y_{3}C^{-2} & 0\\
         T_{3} & Y_{3}C^{-3}B_{2}
     \end{bmatrix}P^{-1},
     \end{equation}
     where $T_{3}=Z_{3}C^{-2}-{Y_{3}C^{-2}B_{1}C^{-1}-Y_{3}C^{-1}B_{1}C^{-2}}-B_{3}C^{-1}Y_{1}C^{-2}$. Now, on comparing Equations~\eqref{eq3.10}, \eqref{3.11}, \eqref{3.12} and \eqref{3.13}, we get
     $$\tilde{A}^{\#}=\widehat{A}^{\#}+\epsilon_2(-\widehat{A}^{\#}\widehat{A}_{0}\widehat{A}^{\#}+(\widehat{A}^{\#})^{2}\widehat{A}_{0}(I-\widehat{A}\widehat{A}^{\#})+(I-\widehat{A}\widehat{A}^{\#})\widehat{A}_{0}(\widehat{A}^{\#})^{2}).$$
\end{proof}

Next, we present an example for computing the HDGGI of the hyper-dual matrix, provided that it exists.
     
 \begin{example}
     Let $\tilde{A}=A_{0}+\epsilon_1 A_{1}+\epsilon_2 A_{2}+ \epsilon_1 \epsilon_2 A_{3}= \begin{bmatrix}
         1 & 0\\
         0 & 0
     \end{bmatrix}+\epsilon_1 \begin{bmatrix}
         1 & -1\\
         1 & 0
     \end{bmatrix}+\epsilon_2 \begin{bmatrix}
         2 & 1\\
         1 & 0
     \end{bmatrix}+\epsilon_1 \epsilon_2 \begin{bmatrix}
         4 & -1\\
         3 & 0
     \end{bmatrix}$. Then, we have
     $(I-A_{0}A_{0}^{\#})=\begin{bmatrix}
         0 &0\\
         0 &1
     \end{bmatrix}$, $A_1A_0^{\#}A_2=\begin{bmatrix}
         2     &1\\
     2     &1
     \end{bmatrix}$ and $A_2A_0^{\#}A_1=\begin{bmatrix}
         2     &-2\\
         1     &-1
     \end{bmatrix}$. On simplification, we get
     \begin{align*}
            (I-A_{0}A_{0}^{\#})A_{1}(I-A_{0}A_{0}^{\#})&= (I-A_{0}A_{0}^{\#})A_{2}(I-A_{0}A_{0}^{\#})\\
                                                       &= {(I-A_{0}A_{0}^{\#})(A_{3}-A_{1}A_{0}^{\#}A_{2}-A_{2}A_{0}^{\#}A_{1})(I-A_{0}A_{0}^{\#})}=0.
        \end{align*}
        Thus, from Theorem \ref{th2.4} $(iv)$, the HDGGI of $\tilde{A}$ exists and 
        $$\tilde{A}^{\#}=\begin{bmatrix}
            1 & 0\\
            0 & 0
        \end{bmatrix}+\epsilon_1 \begin{bmatrix}
            -1 & -1\\
            1 & 0
        \end{bmatrix}+\epsilon_2 \begin{bmatrix}
            -2 & 1\\
            1 & 0
        \end{bmatrix}+\epsilon_1 \epsilon_2 \begin{bmatrix}
            0 &1\\
            {-3} & 0
        \end{bmatrix}.$$
       
 \end{example}

The following corollary follows directly from Theorem \ref{th2.4}.
\begin{corollary}\label{thm3.4}
Let $\tilde{A} = \widehat{A} + \epsilon_2 \widehat{A}_0$ be a hyper-dual matrix such that the dual index of $\widehat{A}$ is 1. If $\tilde{A}^{\#}$ exists and $\widehat{A}\widehat{A}^{\#}\widehat{A}_0=\widehat{A}_0\widehat{A}\widehat{A}^{\#}=\widehat{A}_0$, then
$$\tilde{A}^{\#}=\widehat{A}^{\#}-\epsilon_2 \widehat{A}^{\#}\widehat{A}_0\widehat{A}^{\#}.$$
\end{corollary}

\section{HDGGI solution of the system of linear hyper-dual equations}\label{rm3}
In this section, we consider the system of linear hyper-dual equations \begin{equation}\label{eq4.1}
    \tilde{A}\tilde{x}=\tilde{b},
\end{equation} where $\tilde{A}=\widehat{A}+\epsilon_2 \widehat{A}_{0}$ and $\tilde{b}=\widehat{b}+\epsilon_2\widehat{b}_{0}$. We will show that the hyper-dual group
 inverse possesses the minimal $P$-norm least-squares property under some suitable conditions. First, we define the range and null space of hyper-dual matrix $\tilde{A}=A_{0}+\epsilon_1 A_{1}+\epsilon_2 A_{2}+\epsilon_1 \epsilon_2 A_{3}$ as
 \begin{align*}
    R(\tilde{A})&=\{\tilde{w}\in \tilde{\mathbb{R}}^n: ~\tilde{w}=\tilde{A}\tilde{z},~ \tilde{z}\in \tilde{\mathbb{R}}^n\}\\
    &=\{A_0x_0+\epsilon_1(A_0x_1+A_1x_0)+\epsilon_2(A_0x_2+A_2x_0)+\epsilon_1\epsilon_2(A_0x_3+A_1x_2+A_2x_1+ 
    \\
    ~~&~~~~~~~~~~~~~~~~~~~~~~~~~~~~~~~~~~~~~~~~~~~~~~~~~~~~~~~~~~~~~~~~~~~~~~~~~~~~~~~~~~~A_3x_0), ~x_i\in \mathbb{R}^n\},
\end{align*}
and 
\begin{align*}
    N(\tilde{A})&=\{\tilde{z}\in \tilde{\mathbb{R}}^n:~ 0=\tilde{A}\tilde{z}\}\\
    &=\{x_{0}+\epsilon_1 x_{1}+\epsilon_2 x_{2}+\epsilon_1 \epsilon_2 x_{3}\in \tilde{\mathbb{R}}^n:~A_0x_0=0, A_0x_1+A_1x_0=0,  \\
    &~~~~~~~~~~~~~~~~~~~~A_0x_2+A_2x_0=0,~A_0x_3+A_1x_2+A_2x_1+A_3x_0=0\}.
\end{align*}
Next, we obtain an equivalent condition such that Equation~\eqref{eq4.1} has a solution of the form $\tilde{A}^\#\tilde{b}$. The proof of the result is straightforward and is therefore omitted.
 
\begin{theorem}
    Let $\tilde{A}=\widehat{A}+\epsilon_2\widehat{A}_0$ be a hyper-dual matrix such that the dual index of $\widehat{A}$ is 1. If the hyper-dual group inverse of $\tilde{A}$ exists, then the linear hyper-dual Equation~\eqref{eq4.1} is consistent if and only if $\tilde{A}\tilde{A}^\#\tilde{b}=\tilde{b}$. Moreover, the general solution to \eqref{eq4.1} is given by
    $$\tilde{A}^{\#}\tilde{b}+(I-\tilde{A}\tilde{A}^{\#})\tilde{z},$$
 for an arbitrary $\tilde{z}\in \tilde{\mathbb{R}}^n$.
\end{theorem}

\begin{corollary}
 Let $\tilde{A}=\widehat{A}+\epsilon_2\widehat{A}_0$ be a hyper-dual matrix such that the dual index of $\widehat{A}$ is 1. If the hyper-dual group inverse of $\tilde{A}$ exists, then $\tilde{x}=\tilde{A}^\#\tilde{b}$ is a solution of the linear hyper-dual Equation~\eqref{eq4.1} if and only if  $\widehat{b}=\widehat{A}\widehat{A}^{\#}\widehat{b}$ and $\widehat{b}_{0}=\widehat{A}\widehat{A}^{\#}\widehat{b}_{0}+(I-\widehat{A}\widehat{A}^{\#})\widehat{A}_{0}\widehat{A}^{\#}\widehat{b}$. 
  \end{corollary}
\begin{proof}
    Suppose that $\tilde{x}=\tilde{A}^{\#}\tilde{b}$ is solution of $\tilde{A}\tilde{x}=\tilde{b}$. Then, we have
     $$\widehat{A}\widehat{A}^{\#}\widehat{b}+\epsilon_2 [(\widehat{A}\widehat{A}^{\#}\widehat{b}_{0}+(I-\widehat{A}\widehat{A}^{\#})\widehat{A}_{0}\widehat{A}^{\#})\widehat{b}]=\widehat{b}+\epsilon_2 \widehat{b}_{0}.$$
     Equating the primal and hyper-dual part, we get 
         $$\widehat{b}=\widehat{A}\widehat{A}^{\#}\widehat{b}~\text{and}~\widehat{b}_{0}=\widehat{A}\widehat{A}^{\#}\widehat{b}_{0}+(I-\widehat{A}\widehat{A}^{\#})\widehat{A}_{0}\widehat{A}^{\#}\widehat{b}.$$
         Conversely, if $\widehat{b}=\widehat{A}\widehat{A}^{\#}\widehat{b}$ and $\widehat{b}_{0}=\widehat{A}\widehat{A}^{\#}\widehat{b}_{0}+(I-\widehat{A}\widehat{A}^{\#})\widehat{A}_{0}\widehat{A}^{\#}\widehat{b}$, then $\tilde{x}=\tilde{A}^{\#}\tilde{b}$ is solution of $\tilde{A}\tilde{x}=\tilde{b}$.
\end{proof}

\begin{theorem}
    Let $\tilde{A}=\widehat{A}+\epsilon_2\widehat{A}_0=A_{0}+\epsilon_1 A_{1}+\epsilon_2 A_{2}+\epsilon_1\epsilon_2 A_{3}$ be a hyper-dual matrix such that the dual index of $\widehat{A}$ is 1. If the hyper-dual group inverse of $\tilde{A}$ exists, then $\tilde{A}^{\#}\tilde{b}$ is the unique solution in $R(\tilde{A})$ of the group normal equation, i.e., 
    $$\tilde{A}^2\tilde{x}=\tilde{A}\tilde{b}.$$
\end{theorem}
\begin{proof}
     It is easy to see that $\tilde{A}^\#\tilde{b}  \in  R(\tilde{A}^\#)=R(\tilde{A})$ is a solution of $\tilde{A}^2\tilde{x}=\tilde{A}\tilde{b}.$ Let $\tilde{u}$ be the solution of $\tilde{A}^2\tilde{x}=\tilde{A}\tilde{b}$ in $R(\tilde{A})$ different from $\tilde{A}^\#\tilde{b}$. Then, $\tilde{u}-\tilde{A}^\#\tilde{b} \in R(\tilde{A})$ and $\tilde{A}^2(\tilde{u}-\tilde{A}^\#\tilde{b})=0.$ Pre-multiplying the last equation by $\tilde{A}^\#$, we get $\tilde{A}(\tilde{u}-\tilde{A}^\#\tilde{b})=0$, i.e., $\tilde{u}-\tilde{A}^\#\tilde{b} \in N(\tilde{A})$. Therefore, $\tilde{u}-\tilde{A}^\#\tilde{b} \in R(\tilde{A}) \cap N(\tilde{A})$. Next, we will show that $R(\tilde{A})\cap N(\tilde{A})=\{0\}$, which implies $\tilde{u}=\tilde{A}^{\#}\tilde{b}$ and the uniqueness of $\tilde{A}^{\#}\tilde{b}$ in $R(\tilde{A})$.

     Let $\tilde{w}\in R(\tilde{A})\cap N(\tilde{A})$. Then, $\tilde{w}=\tilde{A}\tilde{x}$, i.e., $\tilde{w}=A_0x_0+\epsilon_1(A_0x_1+A_1x_0)+\epsilon_2(A_0x_2+A_2x_0)+\epsilon_1\epsilon_2(A_0x_3+A_1x_2+A_2x_1+A_3x_0)$. Moreover, from $\tilde{A}\tilde{w}=0$, we have \begin{align}
    A_{0}^{2}x_0=0\label{4.2}\\ A_{0}^{2}x_1+A_0A_1x_0+A_1A_0x_0=0\label{4.3}\\
    A_{0}^{2}x_2+A_0A_2x_0+A_2A_0x_0=0\label{4.4}\\
    A_{0}^{2}x_3+(A_0A_2+A_2A_0)x_1+(A_0A_1+A_1A_0)x_2\notag\\+(A_0A_3+A_1A_2+A_2A_1+A_3A_0)x_0=0\label{4.5}.
   \end{align}
  Since $A_{0}^{\#}$ exists, we have $A_0^2x_0=0$ which implies that $A_0x_0=0$.
    From \eqref{4.3}, we have
  \begin{align*}
    A_{0}^{2}x_1+A_0A_1x_0=0\\
    \implies A_{0}(A_{0}x_1+A_1x_0)=0\\
    \implies A_{0}x_1+A_1x_0 \in N(A_0).
  \end{align*}
   Also, using $(I-A_0A_0^{\#})A_1(I-A_0A_0^{\#})=0$ we can see that 
  \begin{align*}
    A_1x_0&=A_1A_0A_{0}^{\#}x_0-A_0A_{0}^{\#}A_1x_0+A_0A_{0}^{\#}A_1A_0A_{0}^{\#}x_0\\
    \implies A_1x_0&=A_0A_{0}^{\#}A_1x_0. ~~~(\text{since } A_0x_0=0\implies A_0^{\#}x_0=0)
   \end{align*}
  Therefore, $A_{0}x_1+A_1x_0\in R(A_0)$. Thus, $A_{0}x_1+A_1x_0\in R(A_0)\cap N(A_0)=\{0\}$, i.e., 
  \begin{align}\label{4.6}
      A_{0}x_1+A_1x_0=0.
  \end{align}
   Similarly, from \eqref{4.4}, we get
  \begin{align}\label{4.7}
      A_{0}x_2+A_2x_0=0.
   \end{align}
  From \eqref{4.2}, \eqref{4.5}, \eqref{4.6} and \eqref{4.7}, we get
  \begin{align*}
  &A_{0}^{2}x_3+A_0A_2x_1+A_0A_1x_2+A_0A_3x_0=0\\
  \implies &A_{0}(A_{0}x_3+A_2x_1+A_1x_2+A_3x_0)=0\\
  \implies &A_{0}x_3+A_2x_1+A_1x_2+A_3x_0\in N(A_{0}).
  \end{align*}
{For  $x\in N(A_0)$, we have $(I-A_0A^{\#}_0)x=x$. We know that $x_0\in N(A_0)$, thus $(I-A_0A_0^{\#})(A_3-A_1A_0^{\#}A_2-A_2A_0^{\#}A_1)(I-A_0A_0^{\#})x_0=0,$ i.e., $(I-A_0A_0^{\#})(A_3-A_1A_0^{\#}A_2-A_2A_0^{\#}A_1)x_0=0$, so that $A_3x_0=A_1A_0^{\#}A_2x_0+A_2A_0^{\#}A_1x_0+A_0A_0^{\#}A_3x_0-A_0A_0^{\#}A_1A_0^{\#}A_2x_0-A_0A_0^{\#}A_2A_0^{\#}A_1x_0.$
}

{Now, \begin{align*}  &A_{0}x_3+A_2x_1+A_1x_2+A_3x_0\\=&A_{0}x_3+A_2x_1+A_1x_2+
    A_1A_0^{\#}A_2x_0+A_2A_0^{\#}A_1x_0+A_0A_0^{\#}A_3x_0-A_0A_0^{\#}A_1A_0^{\#}A_2x_0\\
    &~~~-A_0A_0^{\#}A_2A_0^{\#}A_1x_0\\
     =&A_{0}x_3+A_2x_1+A_1x_2-A_1A_0^{\#}A_0x_2-A_2A_0^{\#}A_0x_1-A_0A_0^{\#}(A_1A_0^{\#}A_2+A_2A_0^{\#}A_1-A_3)x_0~\\
     &~~~~~~~~~~~~~~~~~~~~~~~~~~~~~~~~~~~~~~~~~~~~~~~~~~~~~~~~~~~~~~~~~~~~~~~~~~~~~~~~~~~~~~~~~~(\text{using }  \eqref{4.6} \text{ and } \eqref{4.7})\\
     =&A_{0}x_3+A_1(I-A_0^{\#}A_0)x_2+A_2(I-A_0^{\#}A_0)x_1-A_0A_0^{\#}(A_1A_0^{\#}A_2+A_2A_0^{\#}A_1-A_3)x_0\\
     =&A_{0}x_3+A_0A_0^{\#}A_1(I-A_0^{\#}A_0)x_2+A_0A_0^{\#}A_2(I-A_0^{\#}A_0)x_1-A_0A_0^{\#}(A_1A_0^{\#}A_2+A_2A_0^{\#}A_1 \\
    &~~-A_3)x_0\in R(A_0)\\&~~~~~~~~~~~~~~~~~~~~~(\text{using } (I-A_0A_0^{\#})A_1(I-A_0A_0^{\#})=0 \text{ and } (I-A_0A_0^{\#})A_2(I-A_0A_0^{\#})=0)\\
  \implies &A_{0}x_3+A_2x_1+A_1x_2+A_3x_0\in N(A_{0})\cap R(A_0)\\
  \implies &A_{0}x_3+A_2x_1+A_1x_2+A_3x_0=0.    
  \end{align*}}
  Thus, $\tilde{w}=A_0x_0+\epsilon_1(A_0x_1+A_1x_0)+\epsilon_2(A_0x_2+A_2x_0)+\epsilon_1\epsilon_2(A_0x_3+A_1x_2+A_2x_1+A_3x_0)=0$. Therefore, 
  $R(\tilde{A})\cap N(\tilde{A})=\{0\}$.
\end{proof}
Zhong and Zhang \cite{dualgroup} introduced the norm on $\widehat{\R}^n$ as:
       $$\|\widehat{x}\|=\sqrt{\|u\|_{2}^2+\|v\|_{2}^2}.$$
       Here, we extend the above norm for hyper-dual vector $\tilde{x}=\widehat{u}+\epsilon_2 \widehat{v}=u_{0}+\epsilon_1 u_{1}+ \epsilon_2 v_{2}+\epsilon_1 \epsilon_2v_{3}$ and define it as
       \begin{equation}\label{eq4.7}
           \|\tilde{x}\|_{h}=\sqrt{ \| \widehat{u}\|^2+\|\widehat{v}\|^{2}}=\sqrt{\|u_{0}\|_{2}^2+\|u_{1}\|_{2}^2+\|v_{2}\|_{2}^2+\|v_{3}\|_{2}^2}.
        \end{equation} 
 It is worth noting that \eqref{eq4.7} indeed defines a norm.
\begin{theorem}
    Suppose that $\tilde{x}, \tilde{y} \in \tilde{\mathbb{R}}^n$. Then, the following hold
    \begin{enumerate}
        \item[(i)] $\|\tilde{x}\|_{h}\geq0$, and $\|\tilde{x}\|_{h}=0$ if and only if $\tilde{x}=0$.
        \item[(ii)]  $\|a\tilde{x}\|_{h}=|a|\|\tilde{x}\|_{h}$, where $a \in \mathbb{R}$.
        \item[(iii)]$\|\tilde{x}+\tilde{y}\|_{h}\leq \|\tilde{x}\|_{h}+\|\tilde{y}\|_{h}$.
    \end{enumerate}
\end{theorem}
\begin{proof} 
  Let $\tilde{x}=\widehat{u}_{1}+\epsilon_2\widehat{v}_{1}, \tilde{y}=\widehat{u}_{2}+\epsilon_2\widehat{v}_{2} \in \tilde{\mathbb{R}}^n$. Then,  
\begin{enumerate}
    \item[(i)] $\|\tilde{x}\|_{h}\geq0$, and $\|\tilde{x}\|_{h}=0$ if and only if $\tilde{x}=0$.
    \item[(ii)] For $a \in \mathbb{R}$, $\|a\tilde{x}\|_{h}=\sqrt{\|a\widehat{u}_{1}\|^2+\|a\widehat{v}_{1}\|^2}=|a|\sqrt{\|\widehat{u}_{1}\|^2+\|\widehat{v}_{1}\|^2}=|a|\|\tilde{x}\|$.
    \item[(iii)] \begin{align*}
\|\tilde{x}+\tilde{y}\|_{h}^2=&\|\widehat{u}_{1}+\widehat{u}_{2}\|^2+\|\widehat{v}_{1}+\widehat{v}_{2}\|^2\\
            \leq&(\|\widehat{u}_{1}\|+\|\widehat{u}_{2}\|)^2+(\|\widehat{v}_{1}\|+\|\widehat{v}_{2}\|)^2\\
            \leq& \|\widehat{u}_{1}\|^2+\|\widehat{u}_{2}\|^2+\|\widehat{v}_{1}\|^2+\|\widehat{v}_{2}\|^2+2\|\widehat{u}_{1}\|\widehat{u}_{2}\|+2\|\widehat{v}_{1}\|\widehat{v}_{2}\|\\
            \leq& \|\widehat{u}_{1}\|^2+\|\widehat{u}_{2}\|^2+\|\widehat{v}_{1}\|^2+\|\widehat{v}_{2}\|^2\\
            &+2\sqrt{(\|\widehat{u}_{1}\|^2+\|\widehat{u}_{2}\|^2)(\|\widehat{v}_{1}\|^2+\|\widehat{v}_{2}\|^2)}\\
            =&(\sqrt{\|\widehat{u}_{1}\|^2+\|\widehat{u}_{2}\|^2}+\sqrt{\|\widehat{v}_{1}\|^2+\|\widehat{v}_{2}\|^2})^2\\
            =&(\|\tilde{x}\|_{h}+\|\tilde{y}\|_{h})^2,\\
            i.e., \|\tilde{x}+\tilde{y}\|_{h}\leq\|\tilde{x}\|_{h}+\|\tilde{y}\|_{h}.
         \end{align*} 
\end{enumerate}
\end{proof}
Now, we define the $P$-norm of a hyper-dual vector $\tilde{x}=u_{0}+\epsilon_1 u_{1}+ \epsilon_2 v_{2}+\epsilon_1 \epsilon_2v_{3}$ as
$$\|\tilde{x}\|_{P}=\sqrt{\|P^{-1}u_{0}\|_{2}^2+\|P^{-1}u_{1}\|_{2}^2+\|P^{-1}v_{2}\|_{2}^2+\|P^{-1}v_{3}\|_{2}^2}.$$
Consider the inconsistent linear system $Ax=b$ such that $A\in \mathbb{R}^{n\times n}$ and $b\in \mathbb{R}^n$. The element $x^*\in \mathbb{R}^n$ is said to be a least-squares solution with respect to $P$-norm if $||Ax^*-b||_P\leq ||Ax-b||_P$. Among the least-squares solution, the one having minimum $P$-norm is said to be the minimal $P$-norm solution. Wei \cite{wei:1998} showed that $A^\#b$ is the minimal $P$-norm solution of the consistent singular linear system $Ax=b$. Analogously, we prove that the hyper-dual group solution $\tilde{A}^\#\tilde b$ is the minimal $P$-norm solution to the hyper-dual linear system $\tilde A\tilde x=\tilde b$. Additionally, we prove that when the same system is inconsistent, $\tilde{A}^\#\tilde b$ is a least-squares solution with respect to a $P$-norm under suitable conditions. The existence of such $P$ for which the hyper-dual group inverse is the minimal $P$-norm least-squares solution is shown in Lemma \ref{lem4.5}.

The following example demonstrates that even when the hyper-dual group inverse exists, $\tilde{A}^{\#}\tilde{b}$ may not have the minimal $P$-norm among the set $\tilde{x}=\tilde{A}^{\#}\tilde{b}+(I-\tilde{A}\tilde{A}^{\#})\tilde{z}$, where $\tilde{z}\in \tilde{\mathbb{R}}^n$.

\begin{example}
Let $\tilde{A}\tilde{x}=\tilde{b}$ be a consistent hyper-dual linear system such that  $\tilde{A}
=
A_{0}+\epsilon_1 A_{1}+\epsilon_2 A_{2}+\epsilon_1 \epsilon_2 A_{3},$ and $\tilde b
= b_{0}+\epsilon b_{1}+\epsilon b_{2}+\epsilon b_{3},$ where
$$A_{0}= \begin{bmatrix}
1 & 0\\
0 & 0
\end{bmatrix}, ~
A_{1}= \begin{bmatrix}
1 & -1\\
1 & 0
\end{bmatrix}, ~A_{2}= \begin{bmatrix}
2 & 1\\
1 & 0
\end{bmatrix}, ~
A_{3}= \begin{bmatrix}
4 & -1\\
3 & 0
\end{bmatrix},$$
and 
$$b_0=\begin{bmatrix}
1\\
0
\end{bmatrix}, b_{1}=
\begin{bmatrix}
1\\
1
\end{bmatrix}, b_{2}=
\begin{bmatrix}
0\\
1
\end{bmatrix}, b_{3}=
\begin{bmatrix}
1\\
1
\end{bmatrix}.$$
Clearly,
\begin{align*} A_0^\#=
\begin{bmatrix}
1&0\\
0&0
\end{bmatrix} \text{ and }
  I-A_0A_0^\#
=
\begin{bmatrix}
0&0\\
0&1
\end{bmatrix}.  
\end{align*}
Now,
$A_1A_0^\#
=\begin{bmatrix}
1&0\\
1&0
\end{bmatrix} \text{ and }A_2A_0^\#
=
\begin{bmatrix}
2&0\\
1&0
\end{bmatrix}.$
Thus,
\begin{align*}
A_1A_0^\#A_2
=
\begin{bmatrix}
2&1\\
2&1
\end{bmatrix} \text{ and }
A_2A_0^\#A_1
=
\begin{bmatrix}
2&-2\\
1&-1
\end{bmatrix}.    
\end{align*}
Further, we have
$$(I-A_0A_0^\#)A_1
=
\begin{bmatrix}
0&0\\
0&1
\end{bmatrix}
\begin{bmatrix}
1&-1\\
1&0
\end{bmatrix}
=
\begin{bmatrix}
0&0\\
1&0
\end{bmatrix}.$$
Then,
$$(I-A_0A_0^\#)A_1(I-A_0A_0^\#)
=
\begin{bmatrix}
0&0\\
1&0
\end{bmatrix}
\begin{bmatrix}
0&0\\
0&1
\end{bmatrix}
=
\begin{bmatrix}
0&0\\
0&0
\end{bmatrix}.$$
Also,
$$(I-A_0A_0^\#)A_2
=
\begin{bmatrix}
0&0\\
0&1
\end{bmatrix}
\begin{bmatrix}
2&1\\
1&0
\end{bmatrix}
=
\begin{bmatrix}
0&0\\
1&0
\end{bmatrix}$$
and
$$(I-A_0A_0^\#)A_2(I-A_0A_0^\#)
=
\begin{bmatrix}
0&0\\
1&0
\end{bmatrix}
\begin{bmatrix}
0&0\\
0&1
\end{bmatrix}
=
\begin{bmatrix}
0&0\\
0&0
\end{bmatrix}.$$
Moreover,
$$A_3-A_1A_0^\#A_2-A_2A_0^\#A_1
=
\begin{bmatrix}
0&0\\
0&0
\end{bmatrix}.$$
Hence,
\begin{align*}
 (I-A_0A_0^\#)A_1(I-A_0A_0^\#)
&=
(I-A_0A_0^\#)A_2(I-A_0A_0^\#)\\
&={(I-A_0A_0^\#)(A_3-A_1A_0^\#A_2-A_2A_0^\#A_1)(I-A_0A_0^\#)
=0.}   
\end{align*}
Therefore, by Theorem \ref{th2.4} (iv), the HDGGI of $\tilde A$ exists and is given by
$$\tilde{A}^{\#}
=
A_0^\#
+\epsilon_1B_1
+\epsilon_2B_2
+\epsilon_1\epsilon_2B_3,$$
where
\begin{align*}
   A_0^\#=
\begin{bmatrix}
1&0\\
0&0
\end{bmatrix},~
B_1=
\begin{bmatrix}
-1&-1\\
1&0
\end{bmatrix},~B_2=
\begin{bmatrix}
-2&1\\
1&0
\end{bmatrix},
\text{ and }
B_3=
\begin{bmatrix}
0&1\\
{-3}&0
\end{bmatrix}. 
\end{align*}
Let $\tilde x^*=\tilde A^\#\tilde b=x_0+\epsilon_1x_1+\epsilon_2x_2+\epsilon_1\epsilon_2x_3$. On computation, we get
$$x_0=A_0^\#b_0
=
\begin{bmatrix}
1\\
0
\end{bmatrix},~~  x_1=A_0^\#b_1+B_1b_0
    =\begin{bmatrix}
0\\
1
\end{bmatrix},~~x_2=A_0^\#b_2+B_2b_0=
\begin{bmatrix}
-2\\
1
\end{bmatrix},$$
and 
$$x_3=A_0^\#b_3+B_1b_2+B_2b_1+B_3b_0=
\begin{bmatrix}
-1\\
-2
\end{bmatrix}.$$
Therefore,
$$\tilde x^*
=
\begin{bmatrix}
1\\
0
\end{bmatrix}
+\epsilon_1
\begin{bmatrix}
0\\
1
\end{bmatrix}
+\epsilon_2
\begin{bmatrix}
-2\\
1
\end{bmatrix}
+\epsilon_1\epsilon_2
\begin{bmatrix}
-1\\
-2
\end{bmatrix}.$$
Taking $ P=I$, we get
$\|\tilde x^*\|_P=2\sqrt{3}.$
Next,
$$\tilde A\tilde A^\#
=
A_0A_0^\#
+\epsilon_1(A_0B_1+A_1A_0^\#)
+\epsilon_2(A_0B_2+A_2A_0^\#).$$
Since
\begin{align*}
A_0A_0^\#
=
\begin{bmatrix}
1&0\\
0&0
\end{bmatrix},~
A_0B_1+A_1A_0^\#
=
\begin{bmatrix}
0&-1\\
1&0
\end{bmatrix}, \text{ and }
A_0B_2+A_2A_0^\#
=
\begin{bmatrix}
0&1\\
1&0
\end{bmatrix}, \end{align*}
it follows that
$$I-\tilde A\tilde A^\#
=
\begin{bmatrix}
0&0\\
0&1
\end{bmatrix}
+\epsilon_1
\begin{bmatrix}
0&1\\
-1&0
\end{bmatrix}
+\epsilon_2
\begin{bmatrix}
0&-1\\
-1&0
\end{bmatrix}.$$
The general solution of $\tilde A\tilde x=\tilde b$ is
$\tilde x
=
\tilde A^\#\tilde b
+
(I-\tilde A\tilde A^\#)\tilde z,~ \text{where } 
\tilde z\in\tilde{\mathbb R}^2.$ 
Assume that
$$\tilde z
=
\begin{bmatrix}
z_1\\
z_2
\end{bmatrix}
+\epsilon_1
\begin{bmatrix}
u_1\\
u_2
\end{bmatrix}
+\epsilon_2
\begin{bmatrix}
v_1\\
v_2
\end{bmatrix}
+\epsilon_1\epsilon_2
\begin{bmatrix}
w_1\\
w_2
\end{bmatrix}.$$
Then, we have
$$(I-\tilde A\tilde A^\#)\tilde z
=
\begin{bmatrix}
0\\
z_2
\end{bmatrix}
+\epsilon_1
\begin{bmatrix}
z_2\\
u_2-z_1
\end{bmatrix}
+\epsilon_2
\begin{bmatrix}
-z_2\\
v_2-z_1
\end{bmatrix}+\epsilon_1\epsilon_2\begin{bmatrix}
v_2-u_2\\
w_2-v_1-u_1
\end{bmatrix}$$ and
$$\tilde x^*
=
\begin{bmatrix}
1\\
0
\end{bmatrix}
+\epsilon_1
\begin{bmatrix}
0\\
1
\end{bmatrix}
+\epsilon_2
\begin{bmatrix}
-2\\
1
\end{bmatrix}
+\epsilon_1\epsilon_2
\begin{bmatrix}
-1\\
-2
\end{bmatrix}.$$
Let $\tilde w=
(I-\tilde A\tilde A^\#)\tilde z.$ Then,
$\tilde x^*+\tilde w
=
\begin{bmatrix}
1\\
z_2
\end{bmatrix}
+\epsilon_1
\begin{bmatrix}
z_2\\
1+u_2-z_2
\end{bmatrix}
+\epsilon_2
\begin{bmatrix}
-2-z_2\\
1+v_2-z_1
\end{bmatrix}
+\epsilon_1\epsilon_2
\begin{bmatrix}
-1+v_2-u_2\\
-2+w_2-v_1-u_1
\end{bmatrix}.$ Taking $P=I,$ the $P$-norm is given as
$$\|\tilde x^*+\tilde w\|_P^2
=1+2z_2^2+(1+u_2-z_1)^2 +(-2-z_2)^2+(1+v_2-z_1)^2 +(-1+v_2-u_2)^2 +(-2+w_2-v_1-u_1)^2.$$ If $z_1=1, z_2=-\dfrac{2}{3}, u_1=0, u_2=-\dfrac{1}{3}, v_1=0, \text{ and } v_2=\dfrac{1}{3}, w_2=2$, then we obtain
\begin{align*}
\|\tilde x^*+\tilde{w}\|_P
&=2.\end{align*}
Therefore, $\|\tilde x^*\|_P^2>
\|\tilde x^*+\tilde{w}\|_P^2.$
Hence 
$\tilde x^*=\tilde A^\#\tilde b$  is not the minimal $P$-norm solution of the hyper-dual linear system $\tilde A\tilde x=\tilde b$.
\end{example}

The next example demonstrates that the group inverse solution may not be a least-squares solution for the hyper-dual linear system.

\begin{example} Let $\tilde{A}\tilde{x}=\tilde{b}$, where $\tilde{b}=\begin{bmatrix}
    0\\1
\end{bmatrix}$ and $\tilde{A}=A_0+\epsilon_1A_1+\epsilon_2A_2+\epsilon_1\epsilon_2A_3$ be such that \begin{align*}
    A_0=\begin{bmatrix}
        1&0\\0&0
    \end{bmatrix}, 
    ~A_1=\begin{bmatrix}
        0&1\\0&0
    \end{bmatrix} ~A_2=\begin{bmatrix}
        0&0\\1&0
    \end{bmatrix} \text{ and }A_3=\begin{bmatrix}
        -1&0\\0&1
    \end{bmatrix},
\end{align*}
i.e.,
\begin{align*}
    \tilde{A}=\begin{bmatrix}
        1-\epsilon_1\epsilon_2&\epsilon_1\\
        \epsilon_2&\epsilon_1\epsilon_2
    \end{bmatrix}.
\end{align*}
Clearly, $ind(\tilde{A})=1$ and $\tilde{A}^{\#}=\tilde{A}$. We observe that our system is inconsistent. We evaluate the group-inverse solution
\begin{align*}
\tilde{x}=\tilde{A}^{\#}\tilde{b}&=\tilde{A}\tilde{b}\\
    &=\begin{bmatrix}
        1-\epsilon_1\epsilon_2&\epsilon_1\\
        \epsilon_2&\epsilon_1\epsilon_2
    \end{bmatrix}\begin{bmatrix}
    0\\1
\end{bmatrix}\\
&=\begin{bmatrix}
    \epsilon_1\\
    \epsilon_1\epsilon_2
\end{bmatrix}.
\end{align*}
Then, the residual vector is
\begin{align*}
    \tilde{r}&=\tilde{b}-\tilde{A}\tilde{x}=\tilde{b}-\tilde{A}(\tilde{A}^{\#}\tilde{b})=\tilde{b}-\tilde{A}(\tilde{A}\tilde{b})=\tilde{b}-\tilde{A}^2\tilde{b}=\tilde{b}-\tilde{A}\tilde{b}\\
    &=\begin{bmatrix}
        0\\1
    \end{bmatrix}-\begin{bmatrix}
    \epsilon_1\\
    \epsilon_1\epsilon_2
\end{bmatrix}=\begin{bmatrix}
    -\epsilon_1\\
    1-\epsilon_1\epsilon_2
\end{bmatrix}.
\end{align*}
We show that the group inverse $\tilde{A}^{\#}\tilde{b}$ fails to minimize the $P$-norm of the residual. If $P=I$, then the  $P$-norm of the residual vector $\tilde{r}$ is
\begin{align*}
    \|\tilde{r}\|_P^2=1^2+(-1)^2+0^2+(-1)^2=3.
\end{align*}
To demonstrate that this is not a least-squares solution, consider the trivial zero vector $\tilde{x}_0 = 0$. Its residual is $\tilde{r}_0 = \tilde{b} - \tilde{A}\tilde{x}_0 = \tilde{b}$. Then,
$\|\tilde{r}_{0}\|_P^2 = 1.$
Hence, we have shown that for $\tilde{x}_0=0$ and $\tilde{x}^*=\tilde{A}^{\#}\tilde{b}$, $\|\tilde{b} - \tilde{A}\tilde{x}_0\|_P^2 < \|\tilde{b} - \tilde{A}\tilde{x}^*\|_P^2$. Therefore, the group-inverse solution is not a least-squares solution with respect to the $P$-norm.
\end{example}

Now, we provide some suitable conditions under which $\tilde{x}=\tilde{A^\#}\tilde{b}$ is the minimal $P$-norm least-squares solution to the hyper-dual linear system. For this, we first establish the existence of such a matrix $P$ in the following lemma.
\begin{lemma}\label{lem4.5}
Let
$
\tilde A
=
\widehat A+\epsilon_2\widehat A_0
=
A_0+\epsilon_1A_1+\epsilon_2A_2+\epsilon_1\epsilon_2A_3$.
Suppose that \(\tilde A^\#\) exists with $
\widehat A^\#
=
A_0^\#-\epsilon_1A_0^\#A_1A_0^\#
$
and
$
\tilde A^\#
=
\widehat A^\#-\epsilon_2\widehat A^\#\widehat A_0\widehat A^\#.
$
Then, there exists a nonsingular matrix \(P\) such that
$
P^{-1}\tilde A P
=
\begin{bmatrix}
\tilde C&0\\
0&0
\end{bmatrix},
$
where
$
\tilde C
=
C_0+\epsilon_1C_1+\epsilon_2C_2+\epsilon_1\epsilon_2C_3.
$
\end{lemma}

\begin{proof}
Since \(A_0^\#\) exists, there exists a nonsingular matrix \(P\) such that
$
P^{-1}A_0P
=
\begin{bmatrix}
C_0&0\\
0&0
\end{bmatrix},
$
where \(C_0\) is nonsingular. Therefore,
$$
P^{-1}A_0^\#P
=
\begin{bmatrix}
C_0^{-1}&0\\
0&0
\end{bmatrix}.
$$
If
$
P^{-1}A_1P
=
\begin{bmatrix}
B_1&B_2\\
B_3&B_4
\end{bmatrix},
$ then by Theorem \ref{cor3.2},
$
A_0A_0^\#A_1=A_1A_0A_0^\#=A_1. 
$
Also,
$
P^{-1}A_0A_0^\#P
=
\begin{bmatrix}
I&0\\
0&0
\end{bmatrix}.
$
Hence,
$$
P^{-1}(A_0A_0^\#A_1)P
=
\begin{bmatrix}
I&0\\
0&0
\end{bmatrix}
\begin{bmatrix}
B_1&B_2\\
B_3&B_4
\end{bmatrix}
=
\begin{bmatrix}
B_1&B_2\\
0&0
\end{bmatrix}.
$$
But \(A_0A_0^\#A_1=A_1\), which implies that
$
B_3=0$ and
$B_4=0.$ Similarly,
$$
P^{-1}(A_1A_0A_0^\#)P
=
\begin{bmatrix}
B_1&B_2\\
0&0
\end{bmatrix}
\begin{bmatrix}
I&0\\
0&0
\end{bmatrix}
=
\begin{bmatrix}
B_1&0\\
0&0
\end{bmatrix}.
$$
Since \(A_1A_0A_0^\#=A_1\), we get
$B_2=0.$
Therefore,
$$
P^{-1}A_1P
=
\begin{bmatrix}
C_1&0\\
0&0
\end{bmatrix},
$$
where \(C_1=B_1\). Consequently,
$$
P^{-1}\widehat A P
=
P^{-1}A_0P+\epsilon_1P^{-1}A_1P=\begin{bmatrix}
C_0&0\\
0&0
\end{bmatrix}
+
\epsilon_1
\begin{bmatrix}
C_1&0\\
0&0
\end{bmatrix}
=
\begin{bmatrix}
\widehat C&0\\
0&0
\end{bmatrix},
$$
where
$
\widehat C=C_0+\epsilon_1C_1.
$ Hence,
$
P^{-1}\widehat A^\#P
=
\begin{bmatrix}
\widehat C^{-1}&0\\
0&0
\end{bmatrix}.
$\\
Now, let
$
P^{-1}\widehat A_0P
=
\begin{bmatrix}
Y_1&Y_2\\
Y_3&Y_4
\end{bmatrix}.
$
Since \(\tilde A^\#\) exists, 
from Theorem \ref{th2.4}, we have
$
(I-\widehat A\widehat A^\#)\widehat A_0(I-\widehat A\widehat A^\#)=0.
$
Thus, 
$$
P^{-1}(I-\widehat A\widehat A^\#)P
=
\begin{bmatrix}
0&0\\
0&I
\end{bmatrix}.
$$
Therefore,
$$
\begin{bmatrix}
0&0\\
0&I
\end{bmatrix}
\begin{bmatrix}
Y_1&Y_2\\
Y_3&Y_4
\end{bmatrix}
\begin{bmatrix}
0&0\\
0&I
\end{bmatrix}
=
\begin{bmatrix}
0&0\\
0&Y_4
\end{bmatrix}
=0.
$$
Hence,
$
Y_4=0.
$ Thus,
$$
P^{-1}\widehat A_0P
=
\begin{bmatrix}
Y_1&Y_2\\
Y_3&0
\end{bmatrix}.
$$
Now, from Theorem \ref{th2.1}, the assumption
$
\tilde A^\#
=
\widehat A^\#-\epsilon_2\widehat A^\#\widehat A_0\widehat A^\#
$ gives $$
(\widehat A^\#)^2\widehat A_0(I-\widehat A\widehat A^\#)
+
(I-\widehat A\widehat A^\#)\widehat A_0(\widehat A^\#)^2
=0.
$$ Since
$
P^{-1}(\widehat A^\#)^2P
=
\begin{bmatrix}
\widehat C^{-2}&0\\
0&0
\end{bmatrix},
$
we get
$$
\begin{bmatrix}
\widehat C^{-2}&0\\
0&0
\end{bmatrix}
\begin{bmatrix}
Y_1&Y_2\\
Y_3&0
\end{bmatrix}
\begin{bmatrix}
0&0\\
0&I
\end{bmatrix}
+
\begin{bmatrix}
0&0\\
0&I
\end{bmatrix}
\begin{bmatrix}
Y_1&Y_2\\
Y_3&0
\end{bmatrix}
\begin{bmatrix}
\widehat C^{-2}&0\\
0&0
\end{bmatrix}
=0,
$$
which further gives
$$
\begin{bmatrix}
0&\widehat C^{-2}Y_2\\
0&0
\end{bmatrix}
+
\begin{bmatrix}
0&0\\
Y_3\widehat C^{-2}&0
\end{bmatrix}=\begin{bmatrix}
0&\widehat C^{-2}Y_2\\
Y_3\widehat C^{-2}&0
\end{bmatrix}
=0.
$$ Therefore,
$
Y_2=0$
and $Y_3=0.
$
Hence,
$
P^{-1}\widehat A_0P
=
\begin{bmatrix}
Y_1&0\\
0&0
\end{bmatrix}.$
Let $Y_1 = C_2+\epsilon_1C_3$, then
$$
P^{-1}\widehat A_0P=P^{-1}A_2P+\epsilon_1P^{-1}A_3P
=
\begin{bmatrix}
C_2+\epsilon_1C_3&0\\
0&0
\end{bmatrix}.
$$
Comparing the primal part and the dual-part, we obtain
$
P^{-1}A_2P
=
\begin{bmatrix}
C_2&0\\
0&0
\end{bmatrix},
$
and
$
P^{-1}A_3P
=
\begin{bmatrix}
C_3&0\\
0&0
\end{bmatrix}.
$ Hence,
$
P^{-1}A_iP
=
\begin{bmatrix}
C_i&0\\
0&0
\end{bmatrix},~ i=1,2,3.
$ Finally, we have
$$
P^{-1}\tilde A P
=
P^{-1}A_0P
+
\epsilon_1P^{-1}A_1P
+
\epsilon_2P^{-1}A_2P
+
\epsilon_1\epsilon_2P^{-1}A_3P.
$$
Substituting the expressions, we get
$$
P^{-1}\tilde A P
=
\begin{bmatrix}
C_0&0\\
0&0
\end{bmatrix}
+
\epsilon_1
\begin{bmatrix}
C_1&0\\
0&0
\end{bmatrix}
+
\epsilon_2
\begin{bmatrix}
C_2&0\\
0&0
\end{bmatrix}
+
\epsilon_1\epsilon_2
\begin{bmatrix}
C_3&0\\
0&0
\end{bmatrix}.
$$
Therefore,
$$
P^{-1}\tilde A P
=
\begin{bmatrix}
\tilde C&0\\
0&0
\end{bmatrix},
$$
where
$
\tilde C
=
C_0+\epsilon_1C_1+\epsilon_2C_2+\epsilon_1\epsilon_2C_3.
$
\end{proof}

The next result provides conditions under which we get the minimal $P$-norm least-squares solution of the hyper-dual linear system.
\begin{theorem} Let $\tilde A = \widehat A+\epsilon_2\widehat A_0 = A_0+\epsilon_1A_1+\epsilon_2A_2+\epsilon_1\epsilon_2A_3 \in \tilde{\mathbb{R}}^{n\times n}, \tilde b\in \tilde{\mathbb{R}}^n$ be such that $\tilde A^{\#}$ exists with $\widehat A^{\#} = A_0^{\#}-\epsilon_1A_0^{\#}A_1A_0^{\#}, $ and $\tilde A^{\#} = \widehat A^{\#}-\epsilon_2\widehat A^{\#}\widehat A_0\widehat A^{\#}.$ Then \(\tilde x^*\) is a $P$-norm least-squares solution of $ \tilde A\tilde x=\tilde b$ if and only if $\tilde x^*$ satisfies the group normal equation $ \tilde A^2\tilde x^*=\tilde A\tilde b.$ Moreover, the hyper-dual group-inverse solution  $\tilde A^\#\tilde b$  is the unique minimal $P$-norm solution of the group normal equation. \end{theorem}

\begin{proof} Since $\widehat A^{\#} = A_0^\#-\epsilon_1A_0^{\#}A_1A_0^{\#}$ and $\tilde A^\# = \widehat A^\#-\varepsilon_2\widehat A^\#\widehat A_0\widehat A^\#,$ from Lemma \ref{lem4.5} there exists a nonsingular matrix $P$ such that $P^{-1}\tilde{A}P = \begin{bmatrix} \tilde{C} & 0\\ 0 & 0 \end{bmatrix},$ where $\tilde{C} = C_0+\epsilon_1C_1+\epsilon_2C_2+\epsilon_1\epsilon_2C_3$ is nonsingular. Hence, we have $P^{-1}\tilde A^\# P = \begin{bmatrix} \tilde C^{-1} & 0\\ 0 & 0 \end{bmatrix}.$ Therefore, $P^{-1}(I-\tilde A\tilde A^\#)P = \begin{bmatrix} 0 & 0\\ 0 & I \end{bmatrix}.$ Now, we can  write $\tilde{b} = \tilde {A}\tilde A^{\#}\tilde{b} + (I-\tilde{A}\tilde{A}^{\#})\tilde{b}$, and thus $\tilde{b}-\tilde{ A}\tilde x = \tilde A\tilde A^{\#}\tilde b-\tilde{A}\tilde x + (I-\tilde{A}\tilde{A}^{\#})\tilde{b}.$ Let $\tilde{u} = \tilde{A}\tilde{A}^{\#}\tilde{b}-\tilde{A}\tilde{x}$ and   $\tilde v = (I-\tilde A\tilde A^\#)\tilde b.$ Then, we  have $\tilde{b}-\tilde{A}\tilde{x} = \tilde{u}+\tilde{v}.$ Let $\tilde u = u_0+\epsilon_1u_1+\epsilon_2u_2+\epsilon_1\epsilon_2u_3$ and $\tilde v = v_0+\epsilon_1v_1+\epsilon_2v_2+\epsilon_1\epsilon_2v_3.$ By the definition of the \(P\)-norm of hyper-dual vectors, 
\begin{align*}
    \|\tilde u+\tilde v\|_P^2 &= \|\tilde u\|_P^2+\|\tilde v\|_P^2 + 2(P^{-1}u_0)^T(P^{-1}v_0) + 2(P^{-1}u_1)^T(P^{-1}v_1)\\
    &~~+ 2(P^{-1}u_2)^T(P^{-1}v_2) + 2(P^{-1}u_3)^T(P^{-1}v_3).
\end{align*}
We now show that the last four terms of the above equation vanish. Since $\tilde u = \tilde A(\tilde A^\#\tilde b-\tilde x),$ we have $P^{-1}\tilde{u} = (P^{-1}\tilde{A} P)P^{-1}(\tilde{A}^{\#}\tilde{b}-\tilde{x}).$  Assume that $P^{-1}(\tilde{A}^{\#}\tilde{b}-\tilde{x})=\begin{bmatrix} \tilde{y_1}\\ \tilde{y_2} \end{bmatrix}$, then, using $P^{-1}\tilde{A} P = \begin{bmatrix} \tilde{C} & 0\\ 0 & 0 \end{bmatrix},$ it follows that $ P^{-1}\tilde{u} = \begin{bmatrix} \tilde{C}\tilde{y_1}\\ 0 \end{bmatrix}.$ Thus, the lower block of $P^{-1}\tilde{u}$ vanishes which further implies that the lower blocks of $P^{-1}u_i$ also vanishes for each \(i=0,1,2,3\).\\
\noindent
On the other hand, $\tilde v = (I-\tilde A\tilde A^\#)\tilde b$, thus
$P^{-1}\tilde v = P^{-1}(I-\tilde A\tilde A^\#)P\,P^{-1}\tilde b.$ Using $P^{-1}(I-\tilde A\tilde A^\#)P = \begin{bmatrix} 0 & 0\\ 0 & I \end{bmatrix}$ and assuming $P^{-1}\tilde{b}=\begin{bmatrix}
    \tilde{z_1}\\
    \tilde{z_2}
\end{bmatrix}$, we obtain  $P^{-1}\tilde v = \begin{bmatrix} 0\\ \tilde{z_2} \end{bmatrix}.$ Thus, for each \(i=0,1,2,3\), the upper blocks of $P^{-1}v_i$ vanish. Consequently, \[ (P^{-1}u_i)^T(P^{-1}v_i)=0, \text{ for } i=0,1,2,3. \]
Now, we have $\|\tilde b-\tilde A\tilde x\|_P^2 = \|\tilde u+\tilde v\|_P^2 = \|\tilde u\|_P^2+\|\tilde v\|_P^2$, i.e., \[ \|\tilde b-\tilde A\tilde x\|_P^2 = \|\tilde A\tilde A^\#\tilde b-\tilde A\tilde x\|_P^2 + \|(I-\tilde A\tilde A^\#)\tilde b\|_P^2. \] Since $\|\tilde A\tilde A^\#\tilde b-\tilde A\tilde x\|_P^2\geq 0,$ we get \[ \|\tilde b-\tilde A\tilde x\|_P^2 \geq \|(I-\tilde A\tilde A^\#)\tilde b\|_P^2. \] The minimum in the above equation is attained if and only if $\tilde A\tilde A^\#\tilde b-\tilde A\tilde x=0,$ i.e., $\tilde A\tilde x = \tilde A\tilde A^\#\tilde b,$ which is equivalent to the group normal equation 
$\tilde A^2\tilde x = \tilde A\tilde b.$
Thus \(\tilde x^*\) is a \(P\)-norm least-squares solution of $\tilde A\tilde x=\tilde b$ if and only if \(\tilde x^*\) satisfies the group normal equation $ \tilde A^2\tilde x^* = \tilde A\tilde b.$\\

Now  we prove the second part of the result. The general solution of the group normal equation is given by $\tilde x = \tilde A^\#\tilde b + (I-\tilde A\tilde A^\#)\tilde z,~\tilde z\in \tilde{\mathbb{R}}^n.$ Let $\tilde w=\tilde A^\#\tilde b,~\tilde q=(I-\tilde A\tilde A^\#)\tilde z.$ Then \[ \tilde x=\tilde w+\tilde q. \] By the same argument as above, we have \[ P^{-1}\tilde w = \begin{bmatrix} \tilde{w_1}\\ 0 \end{bmatrix}, \qquad P^{-1}\tilde q = \begin{bmatrix} 0\\ \tilde{q_2} \end{bmatrix},\quad \text{ for some }  \tilde{w_1}, \tilde{q_2} \in \mathbb{\tilde{R}}^n.
\] Hence, $\|\tilde x\|_P^2 = \|\tilde A^\#\tilde b\|_P^2 + \|(I-\tilde A\tilde A^\#)\tilde z\|_P^2.$ Thus $\|\tilde x\|_P^2 \geq \|\tilde A^\#\tilde b\|_P^2.$ The above equality holds if and only if $(I-\tilde A\tilde A^\#)\tilde z=0.$ Therefore, \(\tilde A^\#\tilde b\) is the unique minimal $P$-norm solution of the group normal equation. 
\end{proof}

\section{Reverse and forward order laws for particular form}\label{rm4}
In this section, we present the reverse-order law and the forward-order law for HDGGI and HDMPGI for particular form $\tilde{A}^{\#}=\widehat{A}^{\#}-\epsilon_2 \widehat{A}^{\#}\widehat{A}_0\widehat{A}^{\#}$ and $\tilde{A}^{\dagger}=\widehat{A}^{\dagger}-\epsilon_2 \widehat{A}^{\dagger}\widehat{A}_0\widehat{A}^{\dagger}$, respectively, of the hyper-dual matrix $\tilde{A}=\widehat{A}+\epsilon_2\widehat{A}_0$. We start this section with an example that shows that reverse and forward-order laws do not always hold for the DGGI. This will also conclude that the reverse-order law and the forward-order law for HDGGI do not hold.

Let $\widehat{A}=\begin{bmatrix}
2&1&3\\0&0&0\\1&1&2
\end{bmatrix}+\epsilon_1 \begin{bmatrix}
2&2&4\\3&-1&2\\-4&-2&-6
\end{bmatrix}$ and $\widehat{B}=\begin{bmatrix}
1&-1&0\\0&0&0\\-1&3&2
\end{bmatrix}+\epsilon_1\begin{bmatrix}
2&-4&3\\0&0&0\\1&-5&6
\end{bmatrix}$ be dual matrices. The DGGI of $\widehat{A}$ and $\widehat{B}$ are given by  $$\widehat{A}^{\#}=\begin{bmatrix}2&-5&-3\\0&0&0\\-1&3&2
\end{bmatrix}+\epsilon_1 \begin{bmatrix}
27&-78&-51\\13&-35&-22\\-21&60&39
\end{bmatrix}$$ and $$\widehat{B}^{\#}=\begin{bmatrix}
1&-1&0\\0&0&0\\1/2&0&1/2
\end{bmatrix}+\epsilon_1 \begin{bmatrix}
-7/2&-3&-3/2\\0&0&0\\-15/4&-19/4&-9/4
\end{bmatrix},$$
respectively. Also, we have 
$$\widehat{A}\widehat{B}= \begin{bmatrix}
    -1&7&6\\0&0&0\\
    -1&5&4
\end{bmatrix}+\epsilon_1\begin{bmatrix}
    5&-13&32\\1&3&4\\6&-28&3
\end{bmatrix}.$$
Therefore, 
$$(\widehat{A}\widehat{B})^{\#}=\begin{bmatrix}
2&-5&-3\\0&0&0\\1/2&-1&-1/2
\end{bmatrix}+\epsilon_1 \begin{bmatrix}
14&-65/2&-31\\11/2&-15&-19/2\\
-31/4&77/4&41/4
\end{bmatrix},$$
$$\widehat{A}^{\#}\widehat{B}^{\#}=\begin{bmatrix}
    1/2&-2&-3/2\\0&0&0\\0&1&1
\end{bmatrix}+{\epsilon}_1\begin{bmatrix}
23/4&-75/4&-87/4\\2&-13&-11\\-11/2&29/2&33/2\end{bmatrix},$$
and
$$\widehat{B}^{\#}\widehat{A}^{\#}=\begin{bmatrix}
2&-5&-3\\0&0&0\\1/2&-1&-1/2
\end{bmatrix}+\epsilon_1 \begin{bmatrix}
17/2&-30&-43/2\\0&0&0\\
-9/4&3&3/4
\end{bmatrix}.$$ It is clear that $(\widehat{A}\widehat{B})^{\#}\neq \widehat{A}^{\#}\widehat{B}^{\#}\neq\widehat{B}^{\#}\widehat{A}^{\#}.$

We now prove a result that provides sufficient conditions for reverse and forward-order laws. For this, let us take
$\tilde{X}=\widehat{A}+\epsilon_2 \widehat{A}_0=A_0+\epsilon_1 A_1+\epsilon_2 A_2+ \epsilon_1 \epsilon_2A_3$ and $\tilde{Y}=\widehat{C}+\epsilon_2 \widehat{C}_0=C_0+\epsilon_1 C_1+\epsilon_2 C_2+ \epsilon_1 \epsilon_2C_3$ be two hyper-dual matrices.

\begin{theorem}
Let $\tilde{X}=\widehat{A}+\epsilon_2 \widehat{A}_0,\tilde{Y}=\widehat{C}+\epsilon_2 \widehat{C}_0$ be hyper-dual matrices such that the dual index of  $\widehat{A}$ and $\widehat{C}$ is 1. Suppose that HDGGI of $\tilde{X}$, $\tilde{Y}$, $\tilde{X}\tilde{Y}$ exist with $\widehat{A}\widehat{A}^{\#}\widehat{A_0}=\widehat{A_0}\widehat{A}\widehat{A}^{\#}=\widehat{A_0}$ and $\widehat{C}\widehat{C}^{\#}\widehat{C_0}=\widehat{C_0}\widehat{C}\widehat{C}^{\#}=\widehat{C_0}$. If $\widehat{A}\widehat{C}=\widehat{C}\widehat{A}$, $\widehat{C}^{\#}\widehat{A}_0=\widehat{A}_0\widehat{C}^{\#}$ and  $\widehat{A}^{\#}\widehat{C}_0=\widehat{C}_0\widehat{A}^{\#}$, then $(\tilde{X}\tilde{Y})^{\#}=\tilde{X}^{\#}\tilde{Y}^{\#}=\tilde{Y}^{\#}\tilde{X}^{\#}$.
\end{theorem}
\begin{proof} By Corollary \ref{thm3.4},  the HDGGI of $\tilde{X}$ and $\tilde{Y}$ are
$\tilde{X}^{\#}=\widehat{A}^{\#}-\epsilon_2 \widehat{A}^{\#}\widehat{A}_0\widehat{A}^{\#}$ and $\tilde{Y}^{\#}=\widehat{C}^{\#}-\epsilon_2 \widehat{C}^{\#}\widehat{C}_0\widehat{C}^{\#}$, respectively. So,  $\tilde{X}^{\#}\tilde{Y}^{\#}=(\widehat{A}^{\#}-\epsilon_2 \widehat{A}^{\#}\widehat{A}_0\widehat{A}^{\#})(\widehat{C}^{\#}-\epsilon_2 \widehat{C}^{\#}\widehat{C}_0\widehat{C}^{\#})=\widehat{A}^{\#}\widehat{C}^{\#}-\epsilon_2 (\widehat{A}^{\#}\widehat{C}^{\#}\widehat{C}_0\widehat{C}^{\#}+\widehat{A}^{\#}\widehat{A}_0\widehat{A}^{\#}\widehat{C}^{\#}).$ Also, $\tilde{X}\tilde{Y}=(\widehat{A}+\epsilon_2 \widehat{A}_0)(\widehat{C}+\epsilon_2 \widehat{C}_0)=\widehat{A}\widehat{C}+\epsilon_2 (\widehat{A}\widehat{C}_0+\widehat{A}_0\widehat{C}).$  Again, using assumptions $\widehat{C}^{\#}\widehat{A}_0=\widehat{A}_0\widehat{C}^{\#}$, $\widehat{A}^{\#}\widehat{C}_0=\widehat{C}_0\widehat{A}^{\#}$ and by Corollary \ref{thm3.4}, $(\tilde{X}\tilde{Y})^{\#}=(\widehat{A}\widehat{C})^{\#}-\epsilon_2 ((\widehat{A}\widehat{C})^{\#}(\widehat{A}\widehat{C}_0+\widehat{A}_0\widehat{C})(\widehat{A}\widehat{C})^{\#})$. The equality $\widehat{A}\widehat{C}=\widehat{C}\widehat{A}$ implies that  $(\widehat{A}\widehat{C})^{\#}=\widehat{A}^{\#}\widehat{C}^{\#}=\widehat{C}^{\#}\widehat{A}^{\#}$. Further, 
\begin{align*}
    (\tilde{X}\tilde{Y})^{\#}&=(\widehat{A}\widehat{C})^{\#}-\epsilon_2 ((\widehat{A}\widehat{C})^{\#}(\widehat{A}\widehat{C}_0+\widehat{A}_0\widehat{C})(\widehat{A}\widehat{C})^{\#})\\
    &=\widehat{A}^{\#}\widehat{C}^{\#}-\epsilon_2 (\widehat{A}^{\#}\widehat{C}^{\#}(\widehat{A}\widehat{C}_0+\widehat{A}_0\widehat{C})\widehat{A}^{\#}\widehat{C}^{\#})\\
    &=\widehat{A}^{\#}\widehat{C}^{\#}-\epsilon_2 (\widehat{A}^{\#}\widehat{C}^{\#}\widehat{A}\widehat{C}_0\widehat{A}^{\#}\widehat{C}^{\#}+\widehat{A}^{\#}\widehat{C}^{\#}\widehat{A}_0\widehat{C}\widehat{A}^{\#}\widehat{C}^{\#})\\
    &=\widehat{A}^{\#}\widehat{C}^{\#}-\epsilon_2 (\widehat{C}^{\#}\widehat{A}^{\#}\widehat{A}\widehat{A}^{\#}\widehat{C}_0\widehat{C}^{\#}+\widehat{A}^{\#}\widehat{A}_0\widehat{C}^{\#}\widehat{C}\widehat{C}^{\#}\widehat{A}^{\#})\\
    &=\widehat{A}^{\#}\widehat{C}^{\#}-\epsilon_2 (\widehat{C}^{\#}\widehat{A}^{\#}\widehat{C}_0\widehat{C}^{\#}+\widehat{A}^{\#}\widehat{A}_0\widehat{C}^{\#}\widehat{A}^{\#})\\
    &=\widehat{A}^{\#}\widehat{C}^{\#}-\epsilon_2 (\widehat{A}^{\#}\widehat{C}^{\#}\widehat{C}_0\widehat{C}^{\#}+\widehat{A}^{\#}\widehat{A}_0\widehat{C}^{\#}\widehat{A}^{\#})\\
    &=\tilde{X}^{\#}\tilde{Y}^{\#}.
\end{align*}

 Following similar steps, we can show that $(\tilde{X}\tilde{Y})^{\#}=\tilde{Y}^{\#}\tilde{X}^{\#}$. Hence, $(\tilde{X}\tilde{Y})^{\#}=\tilde{X}^{\#}\tilde{Y}^{\#}=\tilde{Y}^{\#}\tilde{X}^{\#}$.
 \end{proof}

As a special case of the above result, we next obtain sufficient conditions for reverse and forward-order
laws for the dual group inverse of dual matrices. 
\begin{corollary}
Let $\widehat{X}=A+\epsilon_1 A_0,\widehat{Y}=C+\epsilon_1 C_0$ be dual matrices such that $Ind(A)=Ind(C)=1$. Suppose that DGGI of $\widehat{X}$, $\widehat{Y}$ and $\widehat{X}\widehat{Y}$ exist with ${A}{A}^{\#}{A_0}={A_0}{A}{A}^{\#}={A_0}$ and ${C}{C}^{\#}{C_0}={C_0}{C}{C}^{\#}={C_0}$. If $AC=CA$, $C^{\#}A_0=A_0C^{\#}$ and  $A^{\#}C_0=C_0A^{\#}$, then $(\widehat{X}\widehat{Y})^{\#}=\widehat{X}^{\#}\widehat{Y}^{\#}=\widehat{Y}^{\#}\widehat{X}^{\#}$.
\end{corollary}

Next result provides criteria for the reverse and forward-order laws for the HDMPGI.
\begin{theorem}\label{thm5.3}
Let $\tilde{X}=\widehat{A}+\epsilon_2 \widehat{A}_0,\tilde{Y}=\widehat{C}+\epsilon_2 \widehat{C}_0$ be hyper-dual matrices such that $\widehat{A}^{\dagger}$ and $\widehat{C}^{\dagger}$ exist. Suppose that HDMPGI of $\tilde{X}$, $\tilde{Y}$, $\tilde{X}\tilde{Y}$ exists with $\widehat{A}\widehat{A}^{\dagger}\widehat{A_0}=\widehat{A_0}\widehat{A}\widehat{A}^{\dagger}=\widehat{A_0}$, $\widehat{C}\widehat{C}^{\dagger}\widehat{C_0}=\widehat{C_0}\widehat{C}\widehat{C}^{\dagger}=\widehat{C_0}$, {$\widehat{D}\widehat{D}^\dagger\widehat {E}=\widehat{E}$ and $\widehat E\widehat{D}\widehat{D}^\dagger=\widehat{E}$, where $\widehat{D}=\widehat{A}\widehat{C}$ and $\widehat{E}=\widehat{A}\widehat{C}_0+\widehat{A}_0\widehat{C}$}. If $\widehat{A}\widehat{C}=\widehat{C}\widehat{A}$, $\widehat{A}^T\widehat{C}=\widehat{C}\widehat{A}^T$, $\widehat{C}^{\dagger}\widehat{A}_0=\widehat{A}_0\widehat{C}^{\dagger}$ and  $\widehat{A}^{\dagger}\widehat{C}_0=\widehat{C}_0\widehat{A}^{\dagger}$, then $(\tilde{X}\tilde{Y})^{\dagger}=\tilde{X}^{\dagger}\tilde{Y}^{\dagger}=\tilde{Y}^{\dagger}\tilde{X}^{\dagger}$.
\end{theorem}
\begin{proof} By Corollary \ref{cor2.3},  the HDMPGI of $\tilde{X}$ and $\tilde{Y}$ are 
$\tilde{X}^{\dagger}=\widehat{A}^{\dagger}-\epsilon_2 \widehat{A}^{\dagger}\widehat{A}_0\widehat{A}^{\dagger}$ and $\tilde{Y}^{\dagger}=\widehat{C}^{\dagger}-\epsilon_2 \widehat{C}^{\dagger}\widehat{C}_0\widehat{C}^{\dagger}$, respectively. So,  $\tilde{X}^{\dagger}\tilde{Y}^{\dagger}=(\widehat{A}^{\dagger}-\epsilon_2 \widehat{A}^{\dagger}\widehat{A}_0\widehat{A}^{\dagger})(\widehat{C}^{\dagger}-\epsilon_2 \widehat{C}^{\dagger}\widehat{C}_0\widehat{C}^{\dagger})=\widehat{A}^{\dagger}\widehat{C}^{\dagger}-\epsilon_2 (\widehat{A}^{\dagger}\widehat{C}^{\dagger}\widehat{C}_0\widehat{C}^{\dagger}+\widehat{A}^{\dagger}\widehat{A}_0\widehat{A}^{\dagger}\widehat{C}^{\dagger}).$ Also, $\tilde{X}\tilde{Y}=(\widehat{A}+\epsilon_2 \widehat{A}_0)(\widehat{C}+\epsilon_2 \widehat{C}_0)=\widehat{A}\widehat{C}+\epsilon_2 (\widehat{A}\widehat{C}_0+\widehat{A}_0\widehat{C}).$ Again, by Corollary \ref{cor2.3}, $(\tilde{X}\tilde{Y})^{\dagger}=(\widehat{A}\widehat{C})^{\dagger}-\epsilon_2 ((\widehat{A}\widehat{C})^{\dagger}(\widehat{A}\widehat{C}_0+\widehat{A}_0\widehat{C})(\widehat{A}\widehat{C})^{\dagger})$. The hypothesis $\widehat{A}\widehat{C}=\widehat{C}\widehat{A}$ and $\widehat{A}^T\widehat{C}=\widehat{C}\widehat{A}^T$ imply that  $(\widehat{A}\widehat{C})^{\dagger}=\widehat{A}^{\dagger}\widehat{C}^{\dagger}=\widehat{C}^{\dagger}\widehat{A}^{\dagger}$. In addition,
\begin{align*}
    (\tilde{X}\tilde{Y})^{\dagger} &=(\widehat{A}\widehat{C})^{\dagger}-\epsilon_2 ((\widehat{A}\widehat{C})^{\dagger}(\widehat{A}\widehat{C}_0+\widehat{A}_0\widehat{C})(\widehat{A}\widehat{C})^{\dagger})\\
    &=\widehat{A}^{\dagger}\widehat{C}^{\dagger}-\epsilon_2 (\widehat{A}^{\dagger}\widehat{C}^{\dagger}(\widehat{A}\widehat{C}_0+\widehat{A}_0\widehat{C})\widehat{A}^{\dagger}\widehat{C}^{\dagger})\\
    &=\widehat{A}^{\dagger}\widehat{C}^{\dagger}-\epsilon_2 (\widehat{A}^{\dagger}\widehat{C}^{\dagger}\widehat{A}\widehat{C}_0\widehat{A}^{\dagger}\widehat{C}^{\dagger}+\widehat{A}^{\dagger}\widehat{C}^{\dagger}\widehat{A}_0\widehat{C}\widehat{A}^{\dagger}\widehat{C}^{\dagger})\\
    &=\widehat{A}^{\dagger}\widehat{C}^{\dagger}-\epsilon_2 (\widehat{C}^{\dagger}\widehat{A}^{\dagger}\widehat{A}\widehat{A}^{\dagger}\widehat{C}_0\widehat{C}^{\dagger}+\widehat{A}^{\dagger}\widehat{A}_0\widehat{C}^{\dagger}\widehat{C}\widehat{C}^{\dagger}\widehat{A}^{\dagger})\\
    &=\widehat{A}^{\dagger}\widehat{C}^{\dagger}-\epsilon_2 (\widehat{C}^{\dagger}\widehat{A}^{\dagger}\widehat{C}_0\widehat{C}^{\dagger}+\widehat{A}^{\dagger}\widehat{A}_0\widehat{C}^{\dagger}\widehat{A}^{\dagger})\\
    &=\widehat{A}^{\dagger}\widehat{C}^{\dagger}-\epsilon_2 (\widehat{A}^{\dagger}\widehat{C}^{\dagger}\widehat{C}_0\widehat{C}^{\dagger}+\widehat{A}^{\dagger}\widehat{A}_0\widehat{C}^{\dagger}\widehat{A}^{\dagger})\\
    &=\tilde{X}^{\dagger}\tilde{Y}^{\dagger}.
\end{align*}

 Similarly, we can show that $(\tilde{X}\tilde{Y})^{\dagger}=\tilde{Y}^{\dagger}\tilde{X}^{\dagger}$. Hence, $(\tilde{X}\tilde{Y})^{\dagger}=\tilde{X}^{\dagger}\tilde{Y}^{\dagger}=\tilde{Y}^{\dagger}\tilde{X}^{\dagger}$.
 \end{proof}

Sufficient conditions for reverse
and forward-order laws for the dual Moore-Penrose inverse of dual matrices are provided next.
\begin{corollary}
Let $\widehat{X}=A+\epsilon_1 A_0$, $\widehat{Y}=C+\epsilon_1 C_0$ be dual matrices. Suppose that DMPGI of $\widehat{X}$, $\widehat{Y}$, and $\widehat{X}\widehat{Y}$ exist with ${A}{A}^{\dagger}{A_0}={A_0}{A}{A}^{\dagger}={A_0}$, ${C}{C}^{\dagger}{C_0}={C_0}{C}{C}^{\dagger}={C_0}$, {$DD^\dagger E=E$ and $ E D D^\dagger=E$, where $D=A C$ and $E=A C_0+A_0 C$}. If $AC=CA$, $A^TC=CA^T$, $C^{\dagger}A_0=A_0C^{\dagger}$, and  $A^{\dagger}C_0=C_0A^{\dagger}$, then $(\widehat{X}\widehat{Y})^{\dagger}=\widehat{X}^{\dagger}\widehat{Y}^{\dagger}=\widehat{Y}^{\dagger}\widehat{X}^{\dagger}$.
\end{corollary}

From Theorem 2.7 of \cite{rg5}, existence of $\widehat{A}^{\dagger}$ can be  given in terms of dual index of $\widehat{A}$. Therefore, Theorem \ref{thm5.3} can also be reformulated in terms of the dual index of matrices $\widehat{X}$ and $\widehat{Y}$, and is obtained next.

\begin{theorem}
    Let $\tilde{X}=\widehat{A}+\epsilon_2 \widehat{A}_0,\tilde{Y}=\widehat{C}+\epsilon_2 \widehat{C}_0$ be hyper-dual matrices such that the dual index of $\widehat{A}$ and $\widehat{C}$ is 1. Suppose that HDMPGI of $\tilde{X}$, $\tilde{Y}$, $\tilde{X}\tilde{Y}$ exist with $\widehat{A}\widehat{A}^{\dagger}\widehat{A_0}=\widehat{A_0}\widehat{A}\widehat{A}^{\dagger}=\widehat{A_0}$, $\widehat{C}\widehat{C}^{\dagger}\widehat{C_0}=\widehat{C_0}\widehat{C}\widehat{C}^{\dagger}=\widehat{C_0}$, {$\widehat{D}\widehat{D}^\dagger\widehat {E}=\widehat{E}$ and $\widehat E\widehat{D}\widehat{D}^\dagger=\widehat{E}$, where $\widehat{D}=\widehat{A}\widehat{C}$ and $\widehat{E}=\widehat{A}\widehat{C}_0+\widehat{A}_0\widehat{C}$}. If $\widehat{A}\widehat{C}=\widehat{C}\widehat{A}$, $\widehat{A}^T\widehat{C}=\widehat{C}\widehat{A}^T$, $\widehat{C}^{\dagger}\widehat{A}_0=\widehat{A}_0\widehat{C}^{\dagger}$ and  $\widehat{A}^{\dagger}\widehat{C}_0=\widehat{C}_0\widehat{A}^{\dagger}$, then $(\tilde{X}\tilde{Y})^{\dagger}=\tilde{X}^{\dagger}\tilde{Y}^{\dagger}=\tilde{Y}^{\dagger}\tilde{X}^{\dagger}$.
\end{theorem}

\section{Group inverses of dual matrices of $n$-order}\label{rm5}
 
  For notational simplicity, let   
  a dual matrix be represented as $\widehat{A}^{(1)}=B^{(0)}+\epsilon_1 C^{(0)}$ where  $B^{(0)}, C^{(0)}\in \mathbb{R}^{n\times n}$, $\widehat{A}^{(1)}\in \widehat{\mathbb{R}}^{n\times n}$, then the matrix $\widehat{A}^{(1)}$ 
  is referred to as $1$-order dual matrix that constitute of $0$-order matrices $B^{(0)}$ and $C^{(0)}$. Similarly, the hyper-dual matrix $\widehat{A}^{(2)}=\widehat{A}^{(1)}+\epsilon_{2} \widehat{B}^{(1)}$ can be thought of as the $2$-order dual matrix that constitute of $1$-order matrices $\widehat{A}^{(1)}$ and $\widehat{B}^{(1)}$. More generally, a dual matrix of $n$-order \cite{hyper} is given by
     $$\widehat{A}^{(n)}=\widehat{B}^{(n-1)}+\epsilon_{n} \widehat{C}^{(n-1)},$$
   where $\widehat{B}^{(n-1)}$ and $\widehat{C}^{(n-1)}$ are dual matrices of $(n-1)$-order, and $\epsilon_n$ is a dual unit. In this section, we study an equivalent condition for the existence of the group inverse of dual matrices of $n$-order.



 \begin{theorem}
     Let $\widehat{A}^{(n)}=\widehat{B}^{(n-1)}+\epsilon_{n} \widehat{C}^{(n-1)}$ be an $n$-order dual matrix such that $(\widehat{B}^{(n-1)})^{\#}$ exists. Then, $\widehat{A}^{(n)}$ has a group inverse if and only if  
     $$[(I-\widehat{B}^{(n-1)}(\widehat{B}^{(n-1)})^{\#}]\widehat{C}^{(n-1)}[(I-\widehat{B}^{(n-1)}(\widehat{B}^{(n-1)})^{\#}]=0.$$
     Moreover, if the group generalized inverse of $\widehat{A}^{(n)}$ exists, then 
         $$(\widehat{A}^{(n)})^\#=(\widehat{B}^{(n-1)})^\#+\epsilon_{n}\widehat{Z}^{(n-1)},$$
         where
     \begin{align*}
         \widehat{Z}^{(n-1)}&=-({\widehat{B}^{(n-1)}})^{\#}\widehat{C}^{(n-1)}(\widehat{B}^{(n-1)})^\#+((\widehat{B}^{(n-1)})^{\#})^{2}\widehat{C}^{(n-1)}(I-\widehat{B}^{(n-1)}(\widehat{B}^{(n-1)})^{\#})\\
         &~~+(I-\widehat{B}^{(n-1)}(\widehat{B}^{(n-1)})^{\#})\widehat{C}^{(n-1)}(\widehat{B}^{(n-1)})^{\#})^{2}).
     \end{align*}
 \end{theorem}
 \begin{proof}
    If $\widehat{A}^{(n)}$ has a group inverse $\widehat{X}^{(n)}$, let it be  of the form 
      $$\widehat{X}^{(n)}=\widehat{Y}^{(n-1)}+\epsilon_{n}\widehat{Z}^{(n-1)}.$$
     Clearly, $\widehat{A}^{(n)}$ and $\widehat{X}^{(n)}$ satisfy the following matrix equations
      \begin{align*}
           \widehat{A}^{(n)} \widehat{X}^{(n)} \widehat{A}^{(n)} &= \widehat{A}^{(n)}, \\
    \widehat{X}^{(n)} \widehat{A}^{(n)} \widehat{X}^{(n)} &= \widehat{X}^{(n)}, \\
    \widehat{A}^{(n)} \widehat{X}^{(n)} &= \widehat{X}^{(n)} \widehat{A}^{(n)}.
      \end{align*}
      Now, substituting $\widehat{A}^{(n)}=\widehat{B}^{(n-1)}+\epsilon_{n} \widehat{C}^{(n-1)}$ 
      and $\widehat{X}^{(n)}=\widehat{Y}^{(n-1)}+\epsilon_{n}\widehat{Z}^{(n-1)}$
      into the above three equations gives
      \begin{align*}
          \widehat{B}^{(n-1)} \widehat{Y}^{(n-1)} \widehat{B}^{(n-1)} &= \widehat{B}^{(n-1)}, \\
    \widehat{Y}^{(n-1)} \widehat{B}^{(n-1)} \widehat{Y}^{(n-1)} &= \widehat{Y}^{(n-1)}, \\
    \widehat{B}^{(n-1)} \widehat{Y}^{(n-1)} &= \widehat{Y}^{(n-1)} \widehat{B}^{(n-1)}.
      \end{align*}
      This implies that $\widehat{Y}^{(n-1)}=(\widehat{B}^{(n-1)})^\#$.
      Now, equating the dual parts of both sides of the equation $\widehat{A}^{(n)} \widehat{X}^{(n)} \widehat{A}^{(n)} = \widehat{A}^{(n)}$, we get 
      $$\widehat{C}^{(n-1)} = \widehat{C}^{(n-1)} (\widehat{B}^{(n-1)})^\# \widehat{B}^{(n-1)} + \widehat{B}^{(n-1)} (\widehat{B}^{(n-1)})^\# \widehat{C}^{(n-1)} + \widehat{B}^{(n-1)} \widehat{Z}^{(n-1)} \widehat{B}^{(n-1)},$$
      which is equivalent to 
      \begin{align}\label{eqn6.1}
          \widehat{B}^{(n-1)} \widehat{Z}^{(n-1)} \widehat{B}^{(n-1)}&=\widehat{C}^{(n-1)} - \widehat{C}^{(n-1)} (\widehat{B}^{(n-1)})^\# \widehat{B}^{(n-1)} - \widehat{B}^{(n-1)} (\widehat{B}^{(n-1)})^\# \widehat{C}^{(n-1)}.
      \end{align}
     Pre- and post-multiplying the above equation by $(I-\widehat{B}^{(n-1)} (\widehat{B}^{(n-1)})^\#)$, it reduces to
      $$[(I-\widehat{B}^{(n-1)}(\widehat{B}^{(n-1)})^{\#}]\widehat{C}^{(n-1)}[(I-\widehat{B}^{(n-1)}(\widehat{B}^{(n-1)})^{\#}]=0.$$\\
      \indent Conversely, given that $(\widehat{B}^{(n-1)})^{\#}$ exists, and if $$[(I-\widehat{B}^{(n-1)}(\widehat{B}^{(n-1)})^{\#}]\widehat{C}^{(n-1)}[(I-\widehat{B}^{(n-1)}(\widehat{B}^{(n-1)})^{\#}]=0,$$ 
      then we are required to show that the group inverse of $\widehat{A}^{(n)}$ exists, and the matrix 
      $$\widehat{X}^{(n)}=(\widehat{B}^{(n-1)})^{\#}+\epsilon_{n}\widehat{Z}^{(n-1)}$$
      is a group inverse of $\widehat{A}^{(n)}$, such that
      \begin{align*}
          \widehat{Z}^{(n-1)}&=-(\widehat{B}^{(n-1)})^{\#}\widehat{C}^{(n-1)}(\widehat{B}^{(n-1)})^\#+((\widehat{B}^{(n-1)})^{\#})^{2}\widehat{C}^{(n-1)}(I-\widehat{B}^{(n-1)}(\widehat{B}^{(n-1)})^{\#})\\
         &~~+(I-\widehat{B}^{(n-1)}(\widehat{B}^{(n-1)})^{\#})\widehat{C}^{(n-1)}((\widehat{B}^{(n-1)})^{\#})^{2}.
      \end{align*} 
     On some computations with the three equations of group inverse, we obtain
      \begin{align*}
           \widehat{A}^{(n)} \widehat{X}^{(n)} \widehat{A}^{(n)}&=(\widehat{B}^{(n-1)}+\epsilon_{n} \widehat{C}^{(n-1)})(\widehat{B}^{(n-1)})^{\#}+\epsilon_{n}\widehat{Z}^{(n-1)})(\widehat{B}^{(n-1)}+\epsilon_{n} \widehat{C}^{(n-1)})\\
           &=\widehat{B}^{(n-1)}+\epsilon_{n}(\widehat{B}^{(n-1)} (\widehat{B}^{(n-1)})^\# \widehat{C}^{(n-1)}+\widehat{C}^{(n-1)}(\widehat{B}^{(n-1)}) (\widehat{B}^{(n-1)})^\#\\
           &~~-\widehat{B}^{(n-1)}(\widehat{B}^{(n-1)})^{\#}\widehat{C}^{(n-1)}\widehat{B}^{(n-1)}(\widehat{B}^{(n-1)})^{\#}).
      \end{align*}
      From \eqref{eqn6.1},
      $[(I-\widehat{B}^{(n-1)}(\widehat{B}^{(n-1)})^{\#}]\widehat{C}^{(n-1)}[(I-\widehat{B}^{(n-1)}(\widehat{B}^{(n-1)})^{\#}]=0$ can be equivalently rewritten as 
      $$\widehat{C}^{(n-1)}= \widehat{B}^{(n-1)}(\widehat{B}^{(n-1)})^\# \widehat{C}^{(n-1)}+\widehat{C}^{(n-1)}(\widehat{B}^{(n-1)}) (\widehat{B}^{(n-1)})^\#-\widehat{B}^{(n-1)}(\widehat{B}^{(n-1)})^{\#}\widehat{C}^{(n-1)}\widehat{B}^{(n-1)}(\widehat{B}^{(n-1)})^{\#}.$$ So, { $\widehat{A}^{(n)} \widehat{X}^{(n)} \widehat{A}^{(n)}=\widehat{B}^{(n-1)}+\epsilon_{n} \widehat{C}^{(n-1)}=\widehat{A}^{(n)}$.} Furthermore,
          \begin{align*}
              \widehat{X}^{(n)} \widehat{A}^{(n)} \widehat{X}^{(n)}&= (\widehat{B}^{(n-1)})^\#+\epsilon_{n}((\widehat{B}^{(n-1)})^{\#}\widehat{B}^{(n-1)}\widehat{Z}^{(n-1)}+{(\widehat{B}^{(n-1)})^{\#}\widehat{C}^{(n-1)}(\widehat{B}^{(n-1)})^{\#}}\\
              &~~~~~+\widehat{Z}^{(n-1)}\widehat{B}^{(n-1)}(\widehat{B}^{(n-1)})^{\#})\\
              &=(\widehat{B}^{(n-1)})^\#+\epsilon_{n}(-(\widehat{B}^{(n-1)})^{\#}\widehat{C}^{(n-1)}(\widehat{B}^{(n-1)})^\#+((\widehat{B}^{(n-1)})^{\#})^{2}\widehat{C}^{(n-1)}(I-\widehat{B}^{(n-1)}(\widehat{B}^{(n-1)})^{\#})\\
         &~~~~+(I-\widehat{B}^{(n-1)}(\widehat{B}^{(n-1)})^{\#}\widehat{C}^{(n-1)}((\widehat{B}^{(n-1)})^{\#})^{2}\\
         &=(\widehat{B}^{(n-1)})^{\#}+\epsilon_{n}\widehat{Z}^{(n-1)}\\
         &= \widehat{X}^{(n)}.
           \end{align*}   
     Also, 
           \begin{align*}
               \widehat{A}^{(n)}\widehat{X}^{(n)}&=\widehat{B}^{(n-1)}(\widehat{B}^{(n-1)})^{\#}+\epsilon_{n}[\widehat{B}^{(n-1)}\widehat{Z}^{(n-1)}+\widehat{C}^{(n-1)}(\widehat{B}^{(n-1)})^{\#}]\\
               &=\widehat{B}^{(n-1)}(\widehat{B}^{(n-1)})^{\#}+\epsilon_{n}[-\widehat{B}^{(n-1)}(\widehat{B}^{(n-1)})^{\#}\widehat{C}^{(n-1)}(\widehat{B}^{(n-1)})^{\#}+(\widehat{B}^{(n-1)})^{\#}\widehat{C}^{(n-1)}\\
               &~~-(\widehat{B}^{(n-1)})^{\#}\widehat{C}^{(n-1)}\widehat{B}^{(n-1)}(\widehat{B}^{(n-1)})^{\#}
               +\widehat{C}^{(n-1)}(\widehat{B}^{(n-1)})^{\#}]
           \end{align*}
           and 
           \begin{align*}
               \widehat{X}^{(n)}\widehat{A}^{(n)}&=\widehat{B}^{(n-1)}(\widehat{B}^{(n-1)})^{\#}+\epsilon_{n}[\widehat{Z}^{(n-1)}\widehat{B}^{(n-1)}+(\widehat{B}^{(n-1)})^{\#}\widehat{C}]\\
               &=\widehat{B}^{(n-1)}(\widehat{B}^{(n-1)})^{\#}+\epsilon_{n}[-(\widehat{B}^{(n-1)})^{\#}\widehat{C}^{(n-1)}\widehat{B}^{(n-1)}(\widehat{B}^{(n-1)})^{\#}+\widehat{C}^{(n-1)}(\widehat{B}^{(n-1)})^{\#}\\
               &~~-\widehat{B}^{(n-1)}(\widehat{B}^{(n-1)})^{\#}\widehat{C}^{(n-1)}(\widehat{B}^{(n-1)})^{\#}+(\widehat{B}^{(n-1)})^{\#}\widehat{C}^{(n-1)}].
        \end{align*}
        Thus, $ \widehat{A}^{(n)}\widehat{X}^{(n)}=\widehat{X}^{(n)}\widehat{A}^{(n)}$.
        Therefore, $\widehat{X}^{(n)}$ is a group inverse of $\widehat{A}^{(n)}$.
  \end{proof}

Following similar steps as in Theorem \ref{thm3.2}, one may prove the uniqueness of the group inverse of $\widehat{A}^{(n)}$.

 \begin{theorem}
    Let $\widehat{A}^{(n)}=\widehat{B}^{(n-1)}+\epsilon_n\widehat{C}^{(n-1)}$ be an n-order dual matrix such that $(\widehat{B}^{(n-1)})^{\#}$ exists. If the group inverse of $\widehat{A}^{(n)}$ exists, then it is unique.
\end{theorem} 

\section{Conclusion}
The essential determinations are summarized as follows.
\begin{itemize}
\item The necessary and sufficient conditions for the existence of the HDGGI are established.

\item Uniqueness of HDGGI (whenever it exists) is proved. Characterization of the existence of a group inverse of a hyper-dual
matrix along with a compact formula for its computation is presented.

\item Applications of HDGGI are presented for solving a hyper-dual linear system. Additionally, the reverse and forward-order laws for a particular form of  the HDMPI and  HDGGI have been established.

 \item Finally, we have presented conditions for the existence of the group inverse of $n$-order dual matrices.
\end{itemize}

\section*{Data Availability Statements}
Data sharing not applicable to this article as no datasets were generated or analysed during the current study.

\section*{Conflicts of Interest}
The authors declare that they have no conflict of interest.

\section*{Funding} 
This declaration is “not applicable”.

\section*{Acknowledgements}

The first author acknowledges the support of the National Institute of Technology Raipur,
India.  The fourth author was partially supported by Universidad Nacional de La Pampa, Facultad de Ingenier\'{\i}a (Argentina) [Grant Resol. 172/2024], Universidad Nacional del Sur (Argentina) [Grant PGI 24/ZL22],  Universidad Nacional de R\'{\i}o Cuarto  (Argentina) (Grant Resol. RR 449/2024), and Ministerio de Econom\'{\i}a, Industria y Competitividad (Spain) [Grant Red de Excelencia RED2022- 134176-T].

\end{document}